\title{\large{\textbf{PHASE TRANSITION AND LEVEL-SET PERCOLATION \\ FOR THE GAUSSIAN FREE FIELD}}}
\date{}

\documentclass[a4paper,11pt, leqno]{article}
\usepackage[english]{babel}
\usepackage[a4paper,left=3.5cm,right=3.5cm,top=3.0cm,bottom=3.0cm]{geometry}
\usepackage{amsfonts, amsmath, amssymb, amsthm, mathrsfs}
\usepackage{psfrag}
\usepackage{graphicx}
\usepackage{verbatim}
\usepackage{layout}

\usepackage[hang,flushmargin]{footmisc} 
\renewcommand{\P}[0]{\mathbb{P}}
\DeclareMathOperator{\M}{\mathbb{Z}^3}
\DeclareMathOperator{\E}{\mathbb{E}}

\DeclareMathOperator{\IL}{\mathbb{L}}

\DeclareMathOperator{\F}{\mathcal{F}}
\DeclareMathOperator{\A}{\mathcal{A}}
\DeclareMathOperator{\I}{\mathcal{I}}
\DeclareMathOperator{\T}{\mathcal{T}}
\DeclareMathOperator{\PP}{\mathbf{P}}
\DeclareMathOperator{\LL}{\mathscr{L}}

\providecommand{\norm}[1]{\lVert #1 \rVert}

\numberwithin{equation}{section}

\newtheorem{theorem}{Theorem}[section]
\newtheorem{lemma}[theorem]{Lemma}

\newtheorem{proposition}[theorem]{Proposition}

\newtheorem{cor}[theorem]{Corollary}
\newtheorem*{propbis}{Proposition 2.2'}

\theoremstyle{remark}
\newtheorem*{proof1}{Proof of Theorem \ref{T2.6}}

\newtheorem*{proof5}{Proof of Lemma \ref{L3.6}}
\newtheorem*{proofP}{Proof of Lemma \ref{PeierlsBound}}
\newtheorem*{proof6}{Proof of Corollary \ref{C2.7}}

\theoremstyle{definition}
\newtheorem{remark}[theorem]{Remark}

\addtocounter{section}{-1}

\begin{document}

\maketitle
\thispagestyle{empty}

\begin{center}
\vspace{-0.7cm}
Pierre-Fran\c cois Rodriguez\footnote{Departement Mathematik, ETH Z\"urich, CH-8092 Z\"urich, Switzerland. \\ \indent    This research was supported in part by the grant ERC-2009-AdG 245728-RWPERCRI.} and Alain-Sol Sznitman\footnotemark[\value{footnote}]
\end{center}
\vspace{1.0cm}
\begin{abstract}
\centering
\begin{minipage}{0.9\textwidth}
We consider level-set percolation for the Gaussian free field on $\mathbb{Z}^d$, $d \geq 3$, and prove that, as $h$ varies, there is a non-trivial percolation phase transition of the excursion set above level $h$ for all dimensions $d \geq 3$. So far, it was known that the corresponding critical level $h_*(d)$ satisfies $h_*(d) \geq 0$ for all $d \geq 3$ and that $h_*(3)$ is finite, see \cite{BLM}. We prove here that $h_*(d)$ is finite for all $d \geq 3$. In fact, we introduce a second critical parameter $h_{**} \geq h_*$, show that $h_{**}(d)$ is finite for all $d \geq 3$, and that the connectivity function of the excursion set above level $h$ has stretched exponential decay for all $h> h_{**}$. Finally, we prove that $h_*$ is strictly positive in high dimension. It remains open whether $h_*$ and $h_{**}$ actually coincide and whether $h_* > 0$ for all $d \geq 3$. 
\end{minipage}
\end{abstract}

%\vspace{7cm}
%\begin{flushright}
%July 2012
%\end{flushright}

\newpage
%\layout

\thispagestyle{empty}
\mbox{}
\newpage

\section{Introduction}

In the present work, we investigate level-set percolation for the Gaussian free field on $\mathbb{Z}^d$, $d \geq 3$. This problem has already received much attention in the past, see for instance \cite{LS}, \cite{BLM}, and more recently \cite{G}, \cite{M}. The long-range dependence of the model makes this problem particularly interesting, but also harder to analyze. Here, we prove the existence of a non-trivial \nolinebreak critical \nolinebreak level for all $d \geq 3$, and the positivity of this critical level when $d$ is large enough. Some of our \nolinebreak methods \nolinebreak are inspired by the recent progress in the study of the percolative properties of random interlacements, where a similar long-range dependence occurs, see for instance \cite{RS}, \cite{SS}, \nolinebreak \cite{S3}, \nolinebreak \cite{T}.

\bigskip

We now describe our results and refer to Section \ref{NOTATION} for details. We consider the lattice $\mathbb{Z}^d$, $d \geq 3$, endowed with the usual nearest-neighbor graph structure. Our main object of study is the Gaussian free field on $\mathbb{Z}^d$, with canonical law $\P$ on $\mathbb{R}^{\mathbb{Z}^d}$ such that,
\begin{equation} \label{phi}
\begin{split}
&\text{under $\P$, the canonical field $\varphi$ = $(\varphi_x)_{x \in \mathbb{Z}^d}$ is a centered Gaussian} \\
&\text{field with covariance $\mathbb{E}[\varphi_x \varphi_y] = g (x,y)$}, \text{ for all } x,y \in \mathbb{Z}^d,
\end{split}
\end{equation}
where $g(\cdot, \cdot)$ denotes the Green function of simple random walk on $\mathbb{Z}^d$, see \eqref{GreenFunction}. Note in particular the presence of strong correlations, see \eqref{1.15}. For any \textit{level} $h\in \mathbb{R}$, we introduce the (random) subset of $\mathbb{Z}^d$
\begin{equation} \label{Ephi}
E_{\varphi}^{\geq h}  = \{ x \in \mathbb{Z}^d \ ; \ \varphi_x \geq h \},
\end{equation}
sometimes called \textit{excursion set} (above level $h$). We are interested in the event that the origin lies in an infinite cluster of $E_{\varphi}^{\geq h}$, which we denote by $\{ 0 \stackrel{\geq h}{\longleftrightarrow} \infty\}$, and ask for which values of $h$ this event occurs with positive probability. Since 
\begin{equation} \label{eta}
\eta(h) \stackrel{\text{def.}}{=} \P[0 \stackrel{\geq h}{\longleftrightarrow} \infty]
\end{equation}
is decreasing in $h$, it is sensible to define the critical point for level-set percolation as
\begin{equation}\label{h*}
h_*(d) = \inf \{ h \in \mathbb{R} \ ; \ \eta(h) =0 \} \ \in [-\infty, \infty]
\end{equation}
(with the convention $\inf \emptyset = \infty$). A non-trivial phase transition is then said to occur if $ h_*$ is finite. It is known that $h_* (d) \geq 0$ for all $d \geq 3$ and that $h_*(3) < \infty$ (see \cite{BLM}, Corollary 2 and Theorem 3, respectively; see also the concluding Remark 5.1 in \cite{BLM} to understand why the proof does not easily generalize to all $d \geq 3$). It is also known that when $d \geq 4$, for large $h$, there is no directed percolation inside $E_\varphi^{\geq h}$, see \cite{G}, p. 281 (note that this reference studies the percolative properties of the excursion sets of $|\varphi|$ in place of $\varphi$).

It is not intuitively obvious why $h_*$ should be finite, for it seems a priori conceivable that infinite clusters of $E_{\varphi}^{\geq h}$ could exist for \textit{all} $h >0$ due to the strong nature of the correlations. We show in Corollary \ref{C2.7} that this does not occur and that
\begin{equation} \label{THM1}
h_*(d) < \infty, \qquad \text{for all } d\geq 3.
\end{equation}
In fact, we prove a stronger result in Theorem \ref{T2.6}. We define a second critical parameter
\begin{equation} \label{h**}
h_{**}(d) = \inf \big\{ h \in \mathbb{R} \; ; \; \text{for some $\alpha >0$,} \;  \lim_{L\to \infty} L^\alpha \; \P \big[ B(0,L) \stackrel{\geq h}{\longleftrightarrow} S(0, 2L) \big] = 0 \big\}, 
\end{equation}
where the event $\big\{ B(0,L) \stackrel{\geq h}{\longleftrightarrow} S(0, 2L) \big\}$ refers to the existence of a (nearest-neighbor) path in $E_\varphi^{\geq h}$ connecting $B(0,L)$, the ball of radius $L$ around $0$ in the $\ell^\infty$-norm, to $S(0, 2L)$, the $\ell^\infty$-sphere of radius $2L$ around $0$. It is an easy matter (see Corollary \ref{C2.7} below) to show that
\begin{equation} \label{h*vsh**}
h_* \leq h_{**}.
\end{equation}
Now, we prove in Theorem \ref{T2.6} the stronger statement
\begin{equation} \label{THM2}
h_{**} < \infty, \qquad \text{for all } d\geq 3,
\end{equation}
and then obtain as a by-product that 
\begin{equation} \label{THM2bis}
\text{the connectivity function of $E_\varphi^{\geq h}$ has stretched exponential decay for all $h > h_{**}$}
\end{equation}
(see Theorem \ref{2.6} below for a precise statement). This immediately leads to the important question of whether $h_*$ and $h_{**}$ actually coincide. In case they differ, our results imply a marked transition in the decay of the connectivity function of $E_\varphi^{\geq h}$ at $h=h_{**}$, see Remark \ref{2.8} below.

Our second result concerns the critical level $h_*$ in high dimension. We are able to show in Theorem \ref{T3.3} that 
\begin{equation} \label{THM3}
\text{$h_*$ is strictly positive when $d$ is sufficiently large.}
\end{equation}
This is in accordance with recent numerical evidence, see \cite{M}, Chapter 4. We actually prove a stronger result than \eqref{THM3}. Namely, we show that one can find a positive level $h_0$, such that for large $d$, the restriction of $E_\varphi^{\geq h_0}$ to a thick two-dimensional slab percolates, see above \eqref{THM3bis}. Let us however point out that, by a result of \cite{GKR} (see p. 1151 therein), the restriction of $E_\varphi^{\geq 0}$ to $\mathbb{Z}^2$ (viewed as a subset of $\mathbb{Z}^d$), and, a fortiori, the restriction of $E_\varphi^{\geq h}$ to $\mathbb{Z}^2$, when $h$ is positive, do only contain finite connected components: excursion sets above any non-negative level do not percolate in planes. We refer to Remark \ref{ConcRems} 1) for more on this.

\vspace{0.5cm}

We now comment on the proofs. We begin with \eqref{THM2}. The key ingredient is a certain (static) renormalization scheme very similar to the one developed in Section 2 of \cite{SS} for the problem of percolation of the vacant set left by random interlacements (for a precise definition of this model, see \cite{S1}, Section 1; we merely note that the two ``corresponding'' quantities are $E_\varphi^{\geq h}$ and $\mathcal{V}^u$, the vacant set at level $u \geq 0$). We will be interested in the probability of certain crossing events viewed as functions of $h \in \mathbb{R}$,
\begin{align*}
f_n(h) \ \text{``$=$''} \ \P \big[ 
&\text{ $E_\varphi^{\geq h}$ contains a path from a given block of} \\
&\text{side length $L_n$ to the complement of its $L_n$-neighborhood} \big]
\end{align*}
(see \eqref{phievents} for the precise definition), where $(L_n)_{n\geq 0}$ is a geometrically increasing sequence of length scales, see \eqref{2.1}. Note that by \eqref{Ephi}, $f_n$ is decreasing in $h$. We then explicitly construct an increasing but bounded sequence $(h_n)_{n \geq0}$, with (finite) limit $h_\infty$, such that
\begin{equation} \label{0.1}
 \lim_{n \to \infty} f_n (h_n) = 0.
\end{equation}
This readily implies \eqref{THM1}, since $\eta(h_\infty) \leq f_n(h_\infty) \leq f_n(h_n)$ for all $n \geq 0$, hence $\eta(h_\infty)$ vanishes. By separating combinatorial complexity estimates from probabilistic bounds in $f_n(\cdot)$, see \eqref{2.6.1} and Lemma \ref{L2.1}, we are led to investigate the quantity
\begin{equation*} 
\begin{split}
p_n(h) \ \text{``$=$''}  \ \P \big[ \
&\text{$E_\varphi^{\geq h}$ contains paths connecting each of $2^n$ ``well-separated''} \\
&\text{boxes of side length $L_0$ (within a given box of side length $\sim L_n$)} \\
&\text{to the complement of their respective $L_0$-neighborhoods} \ \big]
\end{split}
\end{equation*}
(see \eqref{2.8} for the precise definition), where the $2^n$ boxes are indexed by the ``leaves'' of a dyadic tree of depth $n$. The key to proving \eqref{0.1} is to provide a suitable induction step relating $p_{n+1}(h_{n+1})$ to $p_n(h_n)$, for all $n \geq 0$, where the increase in parameter $h_n \to h_{n+1}$ allows to dominate the interactions (``sprinkling''). This appears in Proposition \ref{P2.2}. One then makes sure that $p_0(h_0)$ is chosen small enough by picking $h_0$ large, see Theorem \ref{T2.6}. The resulting estimates are fine enough to imply not only that $h_\infty \geq h_*$, but even the stretched exponential decay of the connectivity function of $E_\varphi^{\geq h_\infty}$, thus yielding \eqref{THM2}. The proof of \eqref{THM2bis} then only requires a refinement of this argument. Note that the strategy we have just described is precisely the one used in \cite{SS} for the proof of a similar theorem in the context of random interlacements. We actually also provide a generalization of Proposition \ref{P2.2}, which is of independent interest, but goes beyond what is directly needed here, see Proposition 2.2'. It has a similar spirit to the main renormalization step leading to the decoupling inequalities for random interlacements in \cite{S3}, see Remark \ref{R2.3}.

We now comment on the proof of \eqref{THM3}, which has two main ingredients. The first ingredient is a suitable decomposition of the field $\varphi$ restricted to the subspace $\mathbb{Z}^3$ into the sum of two independent Gaussian fields. The first field has independent components and the second field only acts as a ``perturbation'' when $d$ becomes large, see Lemmas \ref{L3.1} and \ref{L3.2} below. The second ingredient combines the fact that the critical value of Bernoulli site percolation on $\mathbb{Z}^3$ is smaller than $\frac{1}{2}$ (see \cite{CR}), with static renormalization (see Chapter $7$ of \cite{Gr}, \cite{GM}, \cite{Pis}), and a Peierls-type argument to control the perturbation created by the second field. This actually enables us to deduce a stronger result than \eqref{THM3}. Namely, we show in Theorem \ref{T3.3} that one can find a level $h_0 > 0$ and a positive integer $L_0$, such that for large $d$ and all $h \leq h_0$, the excursion set $E_{\varphi}^{\geq h}$ already percolates in the two-dimensional slab
\begin{equation} \label{THM3bis}
 \mathbb{Z}^2 \times [0, 2L_0) \times \{ 0\}^{d-3} \subset \mathbb{Z}^d.
\end{equation}

As already pointed out, some of our proofs employ strategies similar to those developed in the study of the percolative properties of the vacant set left by random interlacements, see \cite{SS}, \cite{S3}. This is not a mere coincidence, as we now explain. Continuous-time random interlacements on $\mathbb{Z}^d$, $d \geq 3$, correspond to a certain Poisson point process of doubly infinite trajectories modulo time-shift, governed by a probability $P$, with a non-negative parameter $u$ playing the role of a multiplicative factor of the intensity measure pertaining to this Poisson point process (the bigger $u$, the more trajectories ``fall'' on $\mathbb{Z}^d$), see \cite{S2}, \cite{S1}. This Poisson gas of doubly infinite trajectories (modulo time-shift) induces a random field of occupation times $(L_{x,u})_{x \in \mathbb{Z}^d}$ (so that the interlacement at level $u$ coincides with $\{ x \in \mathbb{Z}^d; L_{x,u}>0 \}$, whereas the vacant set at level $u$ equals $\{ x \in \mathbb{Z}^d; L_{x,u}=0 \}$). This field is closely linked to the Gaussian free field, as the following isomorphism theorem from \cite{S2} shows:
\begin{equation} \label{isom_THM}
\big( L_{x,u}+ \frac{1}{2}\varphi_x^2 \big)_{x \in \mathbb{Z}^d}, \text{ under $P \otimes \P$, has the same law as } \big( \frac{1}{2}(\varphi_x+\sqrt{2u})^2 \big)_{x \in \mathbb{Z}^d}, \text{ under $\P$}.
\end{equation}
It is tempting to use this identity as a transfer mechanism, and we hope to return to this point elsewhere.

We conclude this introduction by describing the organization of this article. In Section \ref{NOTATION}, we introduce some notation and review some known results concerning simple random walk on $\mathbb{Z}^d$ and the Gaussian free field. Section \ref{NONTRIVIAL PT} is devoted to proving that excursion sets at a high level do not percolate. The main results are Theorem \ref{T2.6} and Corollary \ref{C2.7}. The positivity of the critical level in high dimension and the percolation of excursion sets at low positive level in large enough two-dimensional slabs is established in Theorem \ref{T3.3} of Section \ref{POSLARGEd}.
 \vspace{0.3cm}

One final remark concerning our convention regarding constants: we denote by $c,c',\dots$ positive constants with values changing from place to place. Numbered constants $c_0,c_1,\dots$ are defined at the place they first occur within the text and remain fixed from then on until the end of the article. In Sections \ref{NOTATION} and \ref{NONTRIVIAL PT}, constants will implicitly depend on the dimension $d$. In Section \ref{POSLARGEd} however, constants will be purely numerical (and \emph{independent} of $d$). Throughout the entire article, dependence of constants on additional parameters will appear in the notation.

\section{Notation and some useful facts} \label{NOTATION}

In this section, we introduce some notation to be used in the sequel, and review some known results concerning both simple random walk and the Gaussian free field. 

We denote by $\mathbb{N}= \{0,1,2,\dots\}$ the set of natural numbers, and by $\mathbb{Z}=\{ \dots ,-1,0,1,\dots \}$ the set of integers. We write $\mathbb{R}$ for the set of real numbers, abbreviate $x \wedge y = \min \{x,y\}$ and $x \lor y = \max\{ x,y\}$ for any two numbers $x,y \in \mathbb{R}$, and denote by $[x]$ the integer part of $x$, for any $x \geq 0$. We consider the lattice $\mathbb{Z}^d$, and tacitly assume throughout that $d \geq 3$. On $\mathbb{Z}^d$, we respectively denote by $\vert \cdot \vert$ and $\vert \cdot \vert_\infty$ the Euclidean and $\ell^\infty$-norms. Moreover, for any $x \in \mathbb{Z}^d$ and $r \geq 0$, we let $B(x,r) = \{ y \in \mathbb{Z}^d  ; \ \vert y-x \vert_\infty \leq r \}$ and $S(x,r) = \{ y \in \mathbb{Z}^d  ; \ \vert y-x \vert_\infty = r \}$ stand for the the $\ell^\infty$-ball and $\ell^\infty$-sphere of radius $r$ centered at $x$. Given $K$ and $U$ subsets of $\mathbb{Z}^d$, $K^c = \mathbb{Z}^d \setminus K$ stands for the complement of $K$ in $\mathbb{Z}^d$, $\vert K \vert$ for the cardinality of $K$, $K \subset \subset \mathbb{Z}^d $ means that $\vert K \vert< \infty$, and $d(K,U) = \inf \{ \vert x-y \vert_\infty \ ; \ x \in K , y \in U\}$ denotes the $\ell^\infty$-distance between $K$ and $U$. If $K= \{ x\}$, we simply write $d(x, U)$. Finally, we define the inner boundary of $K$ to be the set $\partial^i K = \{ x \in K ; \ \exists y \in K^c, \vert y-x \vert =1 \}$, and the outer boundary of $K$ as $\partial K = \partial^i (K^c)$. We also introduce the diameter of any subset $K \subset \mathbb{Z}^d$, $\text{diam}(K)$, as its $\ell^\infty$-diameter, i.e. $ \text{diam}(K) = \sup \{ |x-y|_\infty \ ; \ x,y \in K \}$.

We endow $\mathbb{Z}^d$ with the nearest-neighbor graph structure, the edge-set consisting of all pairs of sites $\{ x ,y \}$, $x,y \in \mathbb{Z}^d$, such that $|x-y|=1$. A (nearest-neighbor) path is any sequence of vertices $\gamma= (x_i)_{0 \leq i \leq n}$, where $ n \geq 0 $ and $x_i \in \mathbb{Z}^d$ for all $0 \leq i \leq n$, satisfying $|x_i - x_{i-1}|=1$ for all $ 1\leq i \leq n$. Moreover, two lattice sites $x,y$ will be called $*$-nearest neighbors if $|x-y|_\infty =1$. A $*$-path is defined accordingly. Thus, any site $x \in \mathbb{Z}^d$ has $2d$ nearest neighbors and $3^d -1$ $*$-nearest neighbors.

We now introduce the (discrete-time) simple random walk on $\mathbb{Z}^d$. To this end, we let $W$ be the space of nearest-neighbor $\mathbb{Z}^d$-valued trajectories defined for non-negative times, and let $\mathcal{W}$, $(X_n)_{n \geq 0}$, stand for the canonical $\sigma$-algebra and canonical process on $W$, respectively. Since $d \geq 3$, the random walk is transient. Furthermore, we write $P_x$ for the canonical law of the walk starting at $x \in \mathbb{Z}^d$ and $E_x$ for the corresponding expectation. We denote by $g(\cdot, \cdot)$ the Green function
of the walk, i.e.
\begin{equation}\label{GreenFunction}
g(x,y) = \sum_{n \geq 0} P_x [X_n = y], \qquad \text{for } x,y \in \mathbb{Z}^d,
\end{equation}
which is finite (since $d \geq3$) and symmetric. Moreover, $g(x,y)= g(x-y,0) \stackrel{\text{def.}}{=} g(x-y)$ due to translation invariance. Given $U \subset \mathbb{Z}^d$, we further denote the entrance time in $ U $ by $H_U = \inf \{ n\geq 0 ; X_n \in U \}$, the hitting time of $ U $ by $\widetilde{H}_U = \inf \{ n\geq 1 ; X_n \in U \}$, and the exit time from $ U $ by $T_U = \inf \{ n \geq 0 ; X_n \notin U \}= H_{U^c}$. This allows us to define the Green function $g_{U} (\cdot, \cdot)$ killed outside $U$ as
\begin{equation}\label{GreenFunctionSubK}
g_{U} (x,y) = \sum_{n \geq 0} P_x [X_n = y, \ n < T_U], \qquad \text{for } x,y \in \mathbb{Z}^d.
\end{equation}
It vanishes if $x \notin U$ or $y \notin U$. The relation between $g$ and $g_{U}$ for any $U \subset \mathbb{Z}^d$ is the following (we let $K = U^c$):
\begin{equation} \label{G-GsubK}
g(x,y) = g_{U} (x, y) + E_x [ H_K < \infty , \ g(X_{H_K}, y) ], \qquad \text{for } x,y \in \mathbb{Z}^d.
\end{equation}
The proof of \eqref{G-GsubK} is a mere application of the strong Markov property (at time $H_K$). 

We now turn to a few aspects of potential theory associated to simple random walk. For any $K \subset \subset \mathbb{Z}^d$, we write 
\begin{equation} \label{1.10}
e_{K} (x) = P_x [\widetilde{H}_K = \infty], \qquad x \in K,
\end{equation}
for the equilibrium measure (or escape probability) of $K$, and
\begin{equation}\label{1.11}
\text{cap}(K) = \sum_{x \in K} e_{K} (x)
\end{equation}
for its capacity. It immediately follows from \eqref{1.10} and \eqref{1.11} that the capacity is subadditive, i.e. 
\begin{equation} \label{1.12}
\text{cap}(K\cup K') \leq \text{cap}(K) + \text{cap}(K'), \qquad \text{for all } K,K' \subset \subset \mathbb{Z}^d.
\end{equation}
Moreover, the entrance probability in $K$ may be expressed in terms of $e_{K}(\cdot)$ (see for example \cite{Sp}, Theorem 25.1, p. 300) as
\begin{equation} \label{1.13}
P_x [H_K < \infty] = \sum_{y \in K} g(x,y) \cdot e_{K}(y),
\end{equation}
from which, together with classical bounds on the Green function (c.f. \eqref{1.15} below), one easily obtains (see \cite{SS}, Section 1 for a derivation) the following useful bound for the capacity of a box:
\begin{equation} \label{1.14}
\text{cap}(B(0,L)) \leq c L^{d-2}, \qquad \text{for all } L \geq 1.
\end{equation} 
We next review some useful asymptotics of $g(\cdot)$. Given two functions $f_1,f_2: \mathbb{Z}^d \longrightarrow \mathbb{R}$, we write $ f_1(x) \sim f_2 (x) $, as $ |x| \to \infty $, if they are asymptotic, i.e. if $ \lim_{|x| \to \infty} f_1(x) / f_2(x) = 1 $.
\begin{lemma} \label{L1.3} $(d \geq 3)$
\begin{align}
&g(x) \sim  c |x|^{2-d}, \qquad \text{as } |x| \to \infty.   \label{1.15} \\
&g(0)= 1 + \frac{1}{2d} + o(d^{-1}),\qquad \text{as } d \to \infty.  \label{1.16} \\
&P_0 [\widetilde{H}_{\M} = \infty] = 1 - \frac{7}{2d} + o(d^{-1}), \qquad \text{as } d \to \infty,  \label{1.17}
\end{align}
where $\M$ is viewed as $ \big( \M \times \ \{ 0\}^{d-3} \big) \subset \mathbb{Z}^d$ in \eqref{1.17}.
\end{lemma}

\begin{proof}
For \eqref{1.15}, see \cite{L}, Theorem 1.5.4, for \eqref{1.16}, see \cite{Mo}, pp. 246-247. In order to prove \eqref{1.17}, we assume that $d \geq 6$ and define $\pi: \mathbb{Z}^d \longrightarrow \mathbb{Z}^{d-3}:$ $(x^1,\dots,x^d) \mapsto (x^4,\dots,x^d).$ Then, under $P_0$,
\begin{equation*}
Y_n \stackrel{\text{def.}}{=} \pi \circ X_n, \qquad \text{for all } n \geq 0,
\end{equation*}
is a ``lazy'' walk on $\mathbb{Z}^{d-3}$ starting at the origin. Clearly, $\{ \widetilde{H}_{\M} = \infty \} = \{ \widetilde{H}^{(Y)}_0 = \infty \}$, where $\widetilde{H}^{(Y)}_0$ refers to the first return to $0$ for the walk $Y$. Hence,
\begin{align*}
P_0 [\widetilde{H}_{\M} = \infty] = \big[g^{(Y)}(0) \big]^{-1} = \Big[ \frac{d}{d-3} \cdot g^{(d-3)}(0) \Big]^{-1} \stackrel{\eqref{1.16}}{=}  1 - \frac{7}{2d} + o(d^{-1}),
\end{align*}
as $d \to \infty$, where $g^{(Y)}(\cdot)$ denotes the Green function of $Y$ and $g^{(d-3)}(\cdot)$ that of simple random walk on $\mathbb{Z}^{d-3}$.
\end{proof}

We now turn to the Gaussian free field on $\mathbb{Z}^d$, as defined in \eqref{phi}. Given any subset $K \subset \mathbb{Z}^d$, we frequently write $\varphi_{_K}$ to denote the family $(\varphi_x)_{x \in K}$. For arbitrary $a \in \mathbb{R}$ and $K \subset \subset \mathbb{Z}^d$, we also use the shorthand $\{ \varphi_{\vert_K} > a \}$ for the event $\{ \min \{ \varphi_x  ; \ x \in K\} > a \}$ and similarly $\{ \varphi_{\vert_K} < a \}$ instead of $\{ \max \{ \varphi_x ; \ x \in K\} < a \}$. Next, we introduce certain crossing events for the Gaussian free field. To this end, we first consider the space $\Omega = \{0,1 \}^{\mathbb{Z}^d}$ endowed with its canonical $\sigma$-algebra and define, for arbitrary disjoint subsets $K, K' \subset \mathbb{Z}^d$, the event (subset of $\Omega$)
\begin{equation} \label{1.18}
\begin{split} 
\{ K \longleftrightarrow K' \} = \{&\text{there exists an open path (i.e. along which the} \\
&\text{configuration has value $1$) connecting $K$ and $K'$} \}.
\end{split}
\end{equation} 
For any level $h \in \mathbb{R}$, we write $\Phi^h$ for the measurable map from $\mathbb{R}^{\mathbb{Z}^d}$ into $\Omega$ which sends $\varphi \in \mathbb{R}^{\mathbb{Z}^d}$ to $\big( 1\{ \varphi_x \geq h \}\big)_{x \in \mathbb{Z}^d} \in \Omega$, and define
\begin{equation} \label{1.19}
\{ K \stackrel{\geq h}{\longleftrightarrow} K' \} = ( \Phi^h )^{-1} \big( \{ K \longleftrightarrow K' \} \big)
\end{equation}
(a measurable subset of $\mathbb{R}^{\mathbb{Z}^d}$ endowed with its canonical $\sigma$-algebra $\mathcal{F}$), which is the event that $K$ and $K'$ are connected by a (nearest-neighbor) path in $E_\varphi^{\geq h}$, c.f. \eqref{Ephi}. Denoting by $Q^h$ the image of $\P$ under $\Phi^h$, i.e. the law of $\big( 1\{ \varphi_x \geq h \}\big)_{x \in \mathbb{Z}^d}$ on $\Omega$, we have that $\P [K \stackrel{\geq h}{\longleftrightarrow} K'] = Q^h [ K \longleftrightarrow K'].$ Note that $\{ K \stackrel{\geq h}{\longleftrightarrow} K' \}$ is an increasing event upon introducing on $\mathbb{R}^{\mathbb{Z}^d}$ the natural partial order (i.e. $f \leq f'$ when $f_x \leq f_x'$ for all $x \in \mathbb{Z}^d$).

We proceed with a classical fact concerning conditional distributions for the Gaussian free field on $\mathbb{Z}^d$. We could not find a precise reference in the literature, and include a proof for the Reader's convenience. We first define, for $U \subset \mathbb{Z}^d$, the law $\P^U$ on $\mathbb{R}^{\mathbb{Z}^d}$ of the centered Gaussian field with covariance
\begin{equation} \label{P_K}
\mathbb{E}^U[\varphi_x \varphi_y] = g_{U} (x,y), \qquad \text{ for all } x,y \in \mathbb{Z}^d,
\end{equation}
with $g_{U} (\cdot, \cdot)$ given by \eqref{GreenFunctionSubK}. In particular, $\varphi_x = 0$, $\P^U$-almost-surely, whenever $x \in K= U^c$. We then have

\begin{lemma} \label{L1.2} $\quad$
\medskip

\noindent Let $ \emptyset \neq K \subset \subset \mathbb{Z}^d$, $ U=K^c $ and define $(\widetilde{\varphi}_x)_{x \in \mathbb{Z}^d}$ by
\vspace{-0.5ex}
\begin{equation} \label{mudecomp}
\varphi_x = \widetilde{\varphi}_x + \mu_x, \qquad \text{for } x \in \mathbb{Z}^d,
\end{equation}

\vspace{-0.5ex}
\noindent where $\mu_x$ is the $\sigma(\varphi_x ; x \in K)$-measurable map defined as
\vspace{-0.5ex}
\begin{equation} \label{mu}
\mu_x = E_x [H_{K} < \infty , \varphi_{X_{H_{K}}}] = \sum_{y \in K } P_x [H_{K} < \infty , X_{H_{K}} = y] \cdot \varphi_y, \qquad \text{for } x \in \mathbb{Z}^d.
\end{equation}

\vspace{-1.5ex}
\noindent Then, under $\P$,
\vspace{-0.5ex}
\begin{equation} \label{ind+cond_exps}
(\widetilde{\varphi}_x)_{x \in \mathbb{Z}^d}  \text{ is independent from $\sigma(\varphi_x ; x \in K)$, and distributed as $(\varphi_x)_{x \in \mathbb{Z}^d}$ under $\P^{U}$.}
\end{equation}
\end{lemma}

\begin{proof} Note that for all $x \in K$, $\widetilde{\varphi}_x = 0$ (since $\mu_x = \varphi_x$ for $x \in K$, by \eqref{mu}) and that for all $x \in K$, $\varphi_x = 0$, $\P^{U}$-almost surely. Hence, it suffices to consider $(\widetilde{\varphi}_x)_{x\in U}$. We first show independence. By \eqref{mu}, $(\widetilde{\varphi}_x)_{x \in U}$, $(\varphi_y)_{y \in K}$, are centered and jointly Gaussian. Moreover, they are uncorrelated, since for $x \in U$, $y \in K$,
\begin{equation*}
\E [\widetilde{\varphi}_x \varphi_y] = \E [\varphi_x \varphi_y] - \E[\mu_x \varphi_y] \stackrel{\eqref{phi},\eqref{mu}}{=} g(x,y) - \sum_{z \in K}P_x [H_{K} < \infty , X_{H_{K}} = z] g(z,y) \stackrel{\eqref{G-GsubK}}{=} 0.
\end{equation*}
Thus, $(\widetilde{\varphi}_x)_{x \in U}$, $(\varphi_y)_{y \in K}$, are independent. To conclude the proof of Lemma \ref{L1.2}, it suffices to show that
\begin{equation} \label{L1.2.1}
\E \big[1_A \big( (\widetilde{\varphi}_x)_{x \in U} \big)\big] = \E^{U} \big[ 1_A \big( (\varphi_x)_{x \in U}\big) \big], \qquad \text{for all } A \in \F_U,
\end{equation}
where $ \F_U $ stands for the canonical $\sigma$-algebra on $ \mathbb{R}^U $. Furthermore, choosing some ordering $(x_i)_{i \geq 0}$ of $U$, by Dynkin's Lemma, it suffices to assume that $A$ has the form
\begin{equation} \label{L1.2.2}
A= A_{x_0} \times \cdots \times A_{x_n} \times \mathbb{R}^{U \setminus \{x_0, \dots,x_n \}}, \quad \text{for some } n \geq 0 \text{ and } A_{x_i} \in \mathcal{B}(\mathbb{R}), \  i=0,\dots,n.
\end{equation}
We fix some $A$ of the form \eqref{L1.2.2}, and consider a subset $V$ such that $K\cup\{x_0,\dots,x_n \} \subseteq V \subset \subset \mathbb{Z}^d$ (we will soon let $ V $ increase to $ \mathbb{Z}^d $). We let $P^V_x$, $x \in V$, denote the law of simple random walk on $V$ starting at $x$ killed when exiting $V$ (its Green function corresponds to $g_{V}(\cdot, \cdot)$), and define $\widetilde{\varphi}_x^V$ for $x \in V$ as in \eqref{mudecomp} but with $P_x^V$ replacing $P_x$ in the definition \eqref{mu} of $\mu_x$. It then follows from Proposition 2.3 in \cite{S4} (an analogue of the present lemma for \textit{finite} graphs) that
\begin{equation} \label{L1.2.3}
 \E^V \big[1_{A^V} \big( (\widetilde{\varphi}_x^V)_{x \in V \setminus K} \big)\big] =  \E^{V\setminus K} \big[ 1_{A^V} \big( (\varphi_x)_{x \in V\setminus K}\big) \big],
\end{equation}
with $\P^V$, $\P^{V\setminus K}$ as defined in \eqref{P_K} and $A^V = A_{x_0} \times \cdots \times A_{x_n} \times \mathbb{R}^{V \setminus ( K \cup \{x_0, \dots,x_n \})}$. Letting $V \nearrow \mathbb{Z}^d$, it follows that  $g_{V}(x,y) \nearrow g(x,y)$, $g_{V\setminus K}(x,y) \nearrow g_{U}(x,y)$, hence by dominated convergence that both sides of \eqref{L1.2.3} converge towards the respective sides of \eqref{L1.2.1}, thus completing the proof.
\end{proof}

\begin{remark} \label{R1.3}$\quad$
\smallskip 

\noindent Lemma \ref{L1.2} yields a choice of regular conditional distributions for $(\varphi_x)_{x \in \mathbb{Z}^d}$ conditioned on the variables $(\varphi_x)_{x \in K}$, which is tailored to our future purposes. Namely, $\P$-almost surely,
\begin{equation} \label{phi_cond_exps}
\P \big[ (\varphi_x)_{x \in \mathbb{Z}^d} \in \cdot \  \big\vert (\varphi_x)_{x \in K} \big] = \widetilde{\P} \big[ (\widetilde{\varphi}_x + \mu_x)_{x \in \mathbb{Z}^d} \in \cdot \ \big], 
\end{equation}
where $\mu_x$, $x \in \mathbb{Z}^d$ is given by \eqref{mu}, $\widetilde{\P}$ does not act on $(\mu_x)_{x \in \mathbb{Z}^d}$, and $(\widetilde{\varphi}_x)_{x\in \mathbb{Z}^d}$ is a centered Gaussian field under $\widetilde{\P}$, with $\widetilde{\varphi}_x = 0$, $\widetilde{\P}$-almost surely for $x \in K$. Lemma \ref{L1.2} also provides the covariance structure of this field (namely $ g_U(\cdot,\cdot) $, with $ U=K^c $), but its precise form will be of no importance in what follows. Note that conditioning on $(\varphi_x)_{x \in K}$ produces the (random) \textit{shift} $\mu_x$, which is \textit{linear} in the variables $\varphi_y$, $y \in K$. \hfill $\square$ 
\end{remark}

The explicit form of the conditional distributions in \eqref{phi_cond_exps} readily yields the following result, which can be viewed as a consequence of the FKG-inequality for the free field (see for example \cite{GHM}, Chapter 4).

\begin{lemma} \label{FKG}$\quad$
\smallskip

\noindent Let $\alpha \in \mathbb{R}$, $\emptyset \neq K \subset \subset \mathbb{Z}^d$, and assume $A \in \mathcal{F}$ (the canonical $\sigma$-algebra on $\mathbb{R}^{\mathbb{Z}^d}$) is an increasing event. Then
\begin{equation} \label{FKG_ineq}
\P\big[A  \big| \ \varphi_{|_{K}}=  \alpha \big] \ \leq \ \P \big[A  \big| \ \varphi_{|_{K}} \geq \alpha \big],
\end{equation}
where the left-hand side is defined by the version of the conditional expectation in \eqref{phi_cond_exps}.
\end{lemma}

\noindent Intuitively, augmenting the field can only favor the occurrence of $A$, an increasing event.

\begin{proof}
On the event $\big\{ \varphi_{\vert_{K}} \geq \alpha \big\}$, we have, for $\mu_x$, $x \in \mathbb{Z}^d$, as defined in \eqref{mu}, 
\begin{equation} \label{m}
\mu_x = \sum_{y \in K} \varphi_y P_x [H_{K} < \infty , X_{H_{K}} = y] \geq \alpha P_x[H_{K} < \infty] \stackrel{\text{def.}}{=} m_x(\alpha), \qquad \text{for all } x\in \mathbb{Z}^d,
\end{equation}
with equality instead on the event $\big\{ \varphi_{\vert_{K}} = \alpha \big\}$. Since $A$ is increasing, this yields, with a slight abuse of notation,
\begin{equation*}
\begin{array}{lcl}
\mathbb{P} [A  | \varphi_{\vert_{K}}  =  \alpha] \cdot 1_{\{ \varphi_{\vert_{K}} \geq \alpha \}} \hspace{-1ex} & \stackrel{\eqref{phi_cond_exps}}{=} &  \hspace{-1ex}  \widetilde{ \mathbb{P} } \big[ A \big( (\widetilde{\varphi}_x + m_x(\alpha))_{x \in \mathbb{Z}^d} \big)\ \big] \cdot 1_{\{ \varphi_{\vert_{K}} \geq \alpha \}}  \\
& \leq & \hspace{-1ex}  \widetilde{ \mathbb{P} } \big[ A \big( (\widetilde{\varphi}_x + \mu_x)_{x \in \mathbb{Z}^d} \big)\ \big] \cdot 1_{\{ \varphi_{\vert_{K}} \geq \alpha \}}  \\
& \stackrel{\eqref{phi_cond_exps}}{=} & \hspace{-1ex} \mathbb{P} \big[A  \big| \varphi_{\vert_{K}}\big] \cdot 1_{\{ \varphi_{\vert_{K}} \geq \alpha \}}.
\end{array}
\end{equation*}
Integrating both sides with respect to the probability measure $\nu(\cdot)\stackrel{\text{def.}}{=} \P [ \ \cdot \ | \varphi_{\vert_{K}} \geq \alpha]$, we obtain
\begin{equation*}
\mathbb{P} [A  | \varphi_{\vert_{K}} =  \alpha] \leq E_\nu \big[\mathbb{P} \big[A  | \varphi_{\vert_{K}}\big] \cdot 1_{\{ \varphi_{\vert_{K}} \geq \alpha \}} \big] = \P \big[A  \big| \ \varphi_{|_{K}} \geq \alpha \big].
\end{equation*}
This completes the proof of Lemma \ref{FKG}.
\end{proof}

We now introduce the canonical shift $\tau_z$ on $\mathbb{R}^{\mathbb{Z}^d}$, such that $\tau_z (f)(\cdot)= f(\cdot + z)$, for arbitrary $f \in \mathbb{R}^{\mathbb{Z}^d}$ and $z \in \mathbb{Z}^d$. The measure $\P$ is invariant under $\tau_z$, i.e. $\P[\tau_z^{-1}(A)]= \P[A]$, for all $A \in \F$ (the canonical $\sigma$-algebra on $\mathbb{R}^{\mathbb{Z}^d}$), by translation invariance of $g(\cdot,\cdot)$ (see below \eqref{GreenFunction}), and has the following mixing property:
\begin{equation} \label{MIXING}
\lim_{z \to \infty} \P[A \cap \tau_z^{-1}(B) ] = \P[A] \, \P[B], \qquad \text{for all }A,B \in \F
\end{equation}
(one first verifies \eqref{MIXING} for $A,B$ depending on finitely many coordinates with the help of \eqref{1.15} and the general case follows by approximation, see \cite{CL}, pp.157-158). The following lemma gives a $0$-$1$ law for the probability of existence of an infinite cluster in $E_\varphi^{\geq h}$, the excursion set above level $h \in \mathbb{R}$, c.f. \eqref{Ephi}.   

\begin{lemma} \label{perc_prop}$\quad$
\smallskip

\noindent Let $\Psi(h)= \P[E_{\varphi}^{\geq h} \text{ contains an infinite cluster }]$, for arbitrary $h \in \mathbb{R}$. One then has the following dichotomy:
\begin{equation} \label{dich}
\begin{split}
\Psi (h) = \left \{
\begin{array}{rl}
 0, & \text{ if  } \eta(h)=0, \\
 1, & \text{ if  } \eta(h)>0,
\end{array}
\right.
\end{split}
\end{equation}
where $\eta(h)= \P[0 \stackrel{\geq h}{\longleftrightarrow} \infty]$.  In particular, recalling the definition \eqref{h*} of $h_{*}$, \eqref{dich} implies that $\Psi(h)=1$ for all $h< h_*$, and $\Psi(h)=0$ for all $h>h_*$.
\end{lemma}

\begin{proof}
This  follows by ergodicity, which is itself a consequence of the mixing property \eqref{MIXING}.
\end{proof}

\begin{remark} \label{uniqueness}$\quad$
\smallskip

\noindent When $\Psi(h)=1$, in particular in the supercritical regime $h< h_*$, the infinite cluster in $E_{\varphi}^{\geq h}$ is $\P$-almost surely unique. This follows by Theorem 12.2 in \cite{HJ} (Burton-Keane theorem), because the field $\big( 1\{ \varphi_x \geq h \}\big)_{x \in \mathbb{Z}^d}$ is translation invariant and has the finite energy property (see \cite{HJ}, Definition 12.1). \hfill $\square$
\end{remark}

\section{Non-trivial phase transition} \label{NONTRIVIAL PT}

The main goal of this section is the proof of Theorem \ref{T2.6} below, which roughly states that $h_{**}(d)$ (and hence $h_{*}(d)$, c.f. Corollary \ref{C2.7}) is finite for all $d \geq 3$, and that the connectivity function of $E_\varphi^{\geq h}$, c.f. \eqref{Ephi}, has stretched exponential decay for arbitrary $h > h_{**}$. The proof involves a certain renormalization scheme akin to the one developed in \cite{SS} and \cite{S3} in the context of random interlacements. This scheme will be used to derive recursive estimates for the probabilities of certain crossing events, c.f. Proposition \ref{P2.2}, which can subsequently be propagated inductively, c.f. Proposition \ref{P2.5}. The proper initialization of this induction requires a careful choice of the parameters occurring in the renormalization scheme. The resulting bounds constitute the main tool for the proof of the central Theorem \ref{T2.6}. In addition, an extension of Proposition \ref{P2.2} can be found in Remark \ref{R2.3} 2).

We begin by defining on the lattice $\mathbb{Z}^d$ a sequence of length scales 
\begin{equation} \label{2.1}
L_n = l_0^n L_0, \quad \text{for $n \geq 0$,}
\end{equation}
where $L_0\geq 1$ and $l_0 \geq 100$ are both assumed to be integers and will be specified below. Hence, $L_0$ represents the finest scale and $L_1 < L_2 < \dots$ correspond to increasingly coarse scales. We further introduce renormalized lattices
\begin{equation} \label{2.2}
\IL_n =  L_n \mathbb{Z}^{d} \subset \mathbb{Z}^d, \qquad n \geq 0,
\end{equation}
and note that $ \IL_k \supseteq \IL_n$ for all $0 \leq k \leq n$. To each $x \in \IL_n$, we attach the boxes
\begin{equation} \label{2.3}
B_{n,x} \stackrel{\text{def.}}{=} B_x (L_n), \qquad \text{for }  n \geq 0, \ x \in \IL_n,
\end{equation}
where we define $B_x (L) = x + \big( [0, L ) \cap \mathbb{Z} \big)^d$, the box of side length $L$ \textit{attached to} $x$, for any $x \in \mathbb{Z}^d$ and $L \geq 1$ (not to be confused with $B(x, L)$). Moreover, we let
\begin{equation} \label{2.4}
\widetilde{B}_{n,x} =  \bigcup_{y \in \mathbb{L}_n : \ d(B_{n,y}, B_{n,x}) \leq 1} B_{n,y}, \qquad \ n \geq 0, \ x \in \IL_n,
\end{equation}
so that $\{ B_{n,x} ; \ x \in \IL_n \}$ defines a partition of $\mathbb{Z}^d$ into boxes of side length $L_n$ for all $n \geq 0$, and $\widetilde{B}_{n,x}$, $x \in \IL_n$, is simply the union of $B_{n,x}$ and its $*$-neighboring boxes at level $n$. Moreover, for $n \geq 1$ and $x \in \IL_n$, $B_{n,x}$ is the disjoint union of the $l_0^d$ boxes $\{ B_{n-1,y} ; \  y \in B_{n,x} \cap \ \mathbb{L}_{n-1} \}$ at level $n-1$ it contains. We also introduce the indexing sets
\begin{equation} \label{2.5}
\mathcal{I}_n =  \{ n\} \times \IL_n, \qquad n \geq 0,
\end{equation}
and given $(n,x) \in \I_n$, $n \geq 1$, we consider the sets of labels
\begin{equation} \label{2.6}
\begin{split}
&\mathcal{H}_1(n,x) = \big\{ (n-1,y) \in \I_{n-1}   ; \  B_{n-1,y} \subset B_{n,x} \text{ and }  B_{n-1,y} \cap \  \partial^i  B_{n,x} \neq \emptyset \big\}, \\
&\mathcal{H}_2(n,x) = \big\{ (n-1,y) \in \I_{n-1}  ; \  B_{n-1,y} \cap \big\{ z \in \mathbb{Z}^d ; \ d(z,  B_{n,x} ) = [L_{n}/2]  \big\} \neq \emptyset \big\}.
\end{split}
\end{equation}
Note that for any two indices $(n-1,y_i) \in \mathcal{H}_i(n,x)$, $i=1,2$, we have $\widetilde{B}_{n-1,y_1} \cap \widetilde{B}_{n-1,y_2} = \emptyset$ and $\widetilde{B}_{n-1,y_1} \cup \widetilde{B}_{n-1,y_2} \subset \widetilde{B}_{n,x}$. Finally, given $x \in \IL_n$, $n \geq 0$, we introduce $\Lambda_{n,x}$, a family of subsets $\T$ of $\bigcup_{0 \leq k \leq n} \I_k$ (soon to be thought of as binary trees) defined as
\begin{equation} \label{Lambda_nx}
\begin{split}
\Lambda_{n,x} = \Big\{ \T  \subset \bigcup_{k=0}^{n} \I_k \ ; \
&\T \cap \I_n = ( n,x ) \text{ and every $(k,y) \in \T \cap \I_k$, $0 < k \leq n$, has} \\ 
&\text{two `descendants' $(k-1, y_i(k,y)) \in \mathcal{H}_i(k,y) $, $i=1,2$, such} \\
&\text{that } \T \cap \I_{k-1} = \bigcup_{(k,y) \in \T \cap \I_k} \{ (k-1, y_1(k,y)), \ (k-1, y_2(k,y)) \} \Big\}.
\end{split}
\end{equation}
Hence, any $\T \in \Lambda_{n,x}$ can naturally be identified as a binary tree having root $(n,x) \in \I_n$ and depth $n$. Moreover, the following bound on the cardinality of $\Lambda_{n,x}$ is easily obtained,
\begin{equation} \label{2.6.1}
|\Lambda_{n,x}| \leq (cl_0^{d-1})^2 \cdot (cl_0^{d-1})^{2^2} \cdots (cl_0^{d-1})^{2^n} = (cl_0^{d-1})^{2(2^n-1)} \leq (c_0 l_0^{2(d-1)})^{2^n},
\end{equation}
where $c_0 \geq 1$ is a suitable constant. 

We now consider the Gaussian free field $\varphi = (\varphi_x)_{x \in \mathbb{Z}^d}$ on $\mathbb{Z}^d$ defined in \eqref{phi} and introduce the crossing events (c.f. \eqref{1.19} for the notation)
\begin{equation} \label{phievents}
A_{n,x}^h = \{ B_{n,x} \stackrel{\geq h}{\longleftrightarrow}  \partial^i \widetilde{B}_{n,x} \}, \qquad \text{for }h \in \mathbb{R}, \ n\geq 0, \text{ and } x \in \IL_n.
\end{equation}
Three properties of the events $A_{n,x}^h$ will play a crucial role in what follows. Denoting by $\sigma \big( \varphi_y \; ; \; y \in \widetilde{B}_{n,x}\big)$ the $\sigma$-algebra on $\mathbb{R}^{\mathbb{Z}^d}$ generated by the variables $\varphi_y$, $y \in \widetilde{B}_{n,x}$, we have
\begin{align} 
&A_{n,x}^h \in \sigma \big( \varphi_y \; ; \; y \in \widetilde{B}_{n,x}\big), \label{A_meas} \\
&A_{n,x}^h \text{ is increasing (in $\varphi$)} \label{A_inc} \qquad \text{(see the discussion below \eqref{1.19})}, \\
&A_{n,x}^h \supseteq A_{n,x}^{h'}, \  \text{ for all $h,h' \in \mathbb{R}$} \text{  with  } h \leq h'. \label{A_dec_h} 
\end{align}
Indeed, the property \eqref{A_dec_h} that $A_{n,x}^h$ decreases with $h$ follows since $E^{\geq h}_{\varphi} \supseteq E^{\geq h'}_{\varphi}$ for all $h \leq h'$ by definition, c.f. \eqref{Ephi}. Next, we provide a lemma which separates the combinatorial complexity of the number of crossings in $A_{n,x}^h$ from probabilistic estimates, using $\Lambda_{n,x}$ as introduced in \eqref{Lambda_nx}. This separation will be key in obtaining estimates fine enough to yield the desired stretched exponential decay. Albeit being completely analogous to Lemma 2.1 in \cite{S1}, we repeat its proof, for it comprises an essential geometric observation concerning the events $A_{n,x}^h$. 

\begin{lemma} \label{L2.1}
$(n \geq 0$, $(n,x) \in \I_n$, $h \in \mathbb{R})$
\begin{equation} \label{2.7}
\P[A_{n,x}^h] \leq |\Lambda_{n,x} | \sup_{\T \in \Lambda_{n,x}} \P [A_{\T}^h], \qquad \text{where} \qquad A_{\T}^h = \bigcap_{(0,y) \in  \T \cap \I_0} A_{0,y}^h.
\end{equation}
\end{lemma}
\begin{proof}
We use induction on $n$ to show that
\begin{equation} \label{L2.1.1}
A_{n,x}^h \subseteq \bigcup_{\T \in \Lambda_{n,x}} A_{\T}^h, 
\end{equation}
for all $(n,x) \in \I_n$, from which \eqref{2.7} immediately follows. When $n=0$, \eqref{L2.1.1} is trivial. Assume it holds for all $(n-1,y) \in \I_{n-1}$. For any $(n,x) \in \I_n$, a path in $E^{\geq h}_{\varphi}$ starting in $B_{n,x}$ and ending in $\partial^i \widetilde{B}_{n,x}$ must first cross the box $B_{n-1, y_1}$ for some $(n-1, y_1) \in \mathcal{H}_1(n,x)$, and subsequently $B_{n-1, y_2}$ for some $(n-1, y_2) \in \mathcal{H}_2(n,x)$ before reaching $\partial^i \widetilde{B}_{n,x}$, c.f. Figure \ref{fig1} below. Thus,
\begin{equation*}
A_{n,x}^h \subseteq \bigcup_{\substack{(n-1, y_i) \ \in \  \mathcal{H}_i(n,x) \\ i=1,2} } A_{n-1, y_1}^h \cap A_{n-1,y_2}^h.
\end{equation*}
Upon applying the induction hypothesis to $A_{n-1, y_1}^h$ and $A_{n-1, y_2}^h$ separately, the claim \eqref{L2.1.1} follows.
\end{proof}
\begin{figure} [h!] 
\centering
\psfragscanon
\includegraphics[width=10cm]{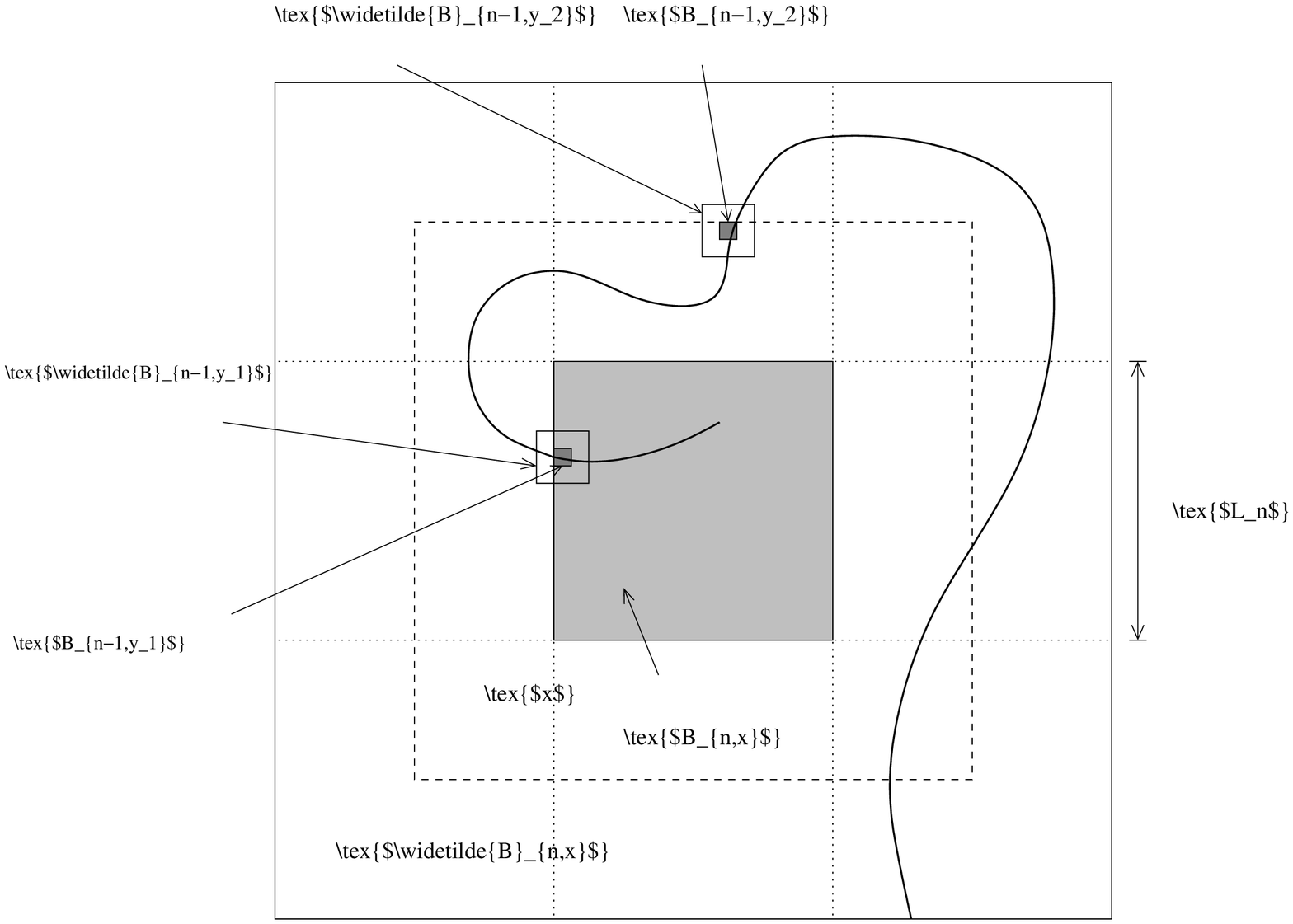} 
\caption{the event $A_{n,x}^h$.}
\label{fig1}
\end{figure}
Before proceeding, we remark that the events $A_{\T}^h$, with $h \in \mathbb{R}$ and $\T \in \Lambda_{n,x}$ for some $(n,x) \in \I_n$, $n \geq0$, defined in \eqref{2.7} inherit certain properties from the events $A_{0,y}^h$, $(0,y) \in \T \cap \I_0$. Namely, it follows from \eqref{A_inc} and \eqref{A_dec_h} that
\begin{equation} \label{A_T_inc}
A_{\T}^h  \text{ is an increasing event (in $\varphi$)},
\end{equation}
and that, for any two levels $h,h' \in \mathbb{R}$,
\begin{equation} \label{A_T_dec_h}
A_{\T}^h \supseteq A_{\T}^{h'} \text{ whenever } h \leq h'. 
\end{equation}
Further, given any $n \geq0$, $(n,x) \in \I_n$, and $\T \in \Lambda_{n,x}$, we define the set
\begin{equation} \label{K_T}
K_{\T} = \bigcup_{(0,y) \ \in \ \T \cap \I_0} \widetilde{B}_{0,y}.
\end{equation}

\vspace{-0.8ex}
\noindent Hence, $K_{\T}$ is the disjoint union of $2^n$ boxes of side length $3L_0$ each, and $K_{\T} \subset \widetilde{B}_{n,x}$. It then immediately follows from the definition of $A_{\T}^h$ in \eqref{2.7} that (see above \eqref{A_meas} for the notation $\sigma(\; \cdot \;)$)
\begin{equation} \label{Ameas}
A_{\T}^h \in \sigma \big( \varphi_y \ ; \ y \in K_{\T} \big).
\end{equation}
Finally, upon introducing
\begin{equation} \label{2.8}
p_n(h) = \sup_{\T \in \Lambda_{n,x}} \P [A_{\T}^h], \qquad \text{for } (n,x) \in \I_n, \ n\geq 0,
\end{equation}

\vspace{-0.8ex}
\noindent which is well-defined (i.e. independent of $x \in \IL_n$) by translation invariance, we obtain $p_n(h) \geq p_n(h')$ whenever $h \leq h'$, by \eqref{A_T_dec_h}. Note also that 
\begin{equation}\label{p0}
p_0(h) = \P \big[ B_{0,x=0} \stackrel{\geq h}{\longleftrightarrow} \partial^i \widetilde{B}_{0,x=0}\big].
\end{equation}
We now derive the aforementioned ``recursive bounds'' for the probabilities $p_n(h_n)$, c.f. \eqref{2.11} below, along a suitable increasing sequence $(h_n)_{n \geq 0}$ (one-step renormalization). These estimates will be key in proving Theorem \ref{T2.6} below.

\begin{proposition} \label{P2.2} $(L_0 \geq 1$, $l_0 \geq 100)$ 
\medskip

\noindent There exist positive constants $c_1$ and $c_2$ such that, defining 
\vspace{-0.5ex}
\begin{equation} \label{2.8.1}
M(n,L_0)= c_2 \big( \log(2^n (3L_0)^d) \big)^{1/2},
\end{equation}

\vspace{-0.5ex}
\noindent then, given any positive sequence $(\beta_n)_{n \geq 0}$ satisfying
\vspace{-0.5ex}
\begin{equation} \label{2.9}
\beta_n \geq (\log 2)^{1/2} + M(n,L_0), \qquad \text{for all } n \geq 0,
\end{equation}

\vspace{-0.5ex}
\noindent and any increasing, real-valued sequence $(h_n)_{n \geq 0}$ satisfying
\vspace{-0.5ex}
\begin{equation} \label{2.10}
h_{n+1} \geq h_n + c_1 \beta_n \big(2 l_0^{-(d-2)} \big)^{n+1}, \qquad \text{for all } n \geq 0,
\end{equation}

\vspace{-0.7ex}
\noindent one has
\vspace{-0.7ex}
\begin{equation} \label{2.11}
p_{n+1}(h_{n+1}) \leq p_n(h_n)^2 + 3  e^{-(\beta_n -M(n,L_0))^2}, \qquad \text{for all } n \geq 0.
\end{equation}
\end{proposition}

The main idea of the proof is to ``decouple'' the event $A_{\T'}^{h_n} \cap A_{\T''}^{h_n}$, where $\T'$ and $\T''$ are the (binary) subtrees at level $n$ of some given tree $\T \in \Lambda_{n+1,x}$, $x \in \IL_{n+1}$, using the increase in parameter $h_n \to h_{n+1}$ to dominate the interactions (``sprinkling'').

\begin{proof}
We let $n \geq 0$, consider some $m = (n+1,x) \in \I_{n+1}$ and some tree $\T \in \Lambda_m$. We decompose 
\begin{equation} \label{Tdec}
\T = \{ m \} \cup \T_{n,y_1(m)} \cup \T_{n,y_2(m)}, 
\end{equation}
where $(n,y_i(m))$, $i=1,2$ are the two descendants of $m$ in $\T$ and
\begin{equation} \label{Ti}
\T_{n,y_i(m)} = \{ (k,z) \in \T : \ \widetilde{B}_{k,z} \subseteq \widetilde{B}_{n,y_i(m)} \}, \quad \text{for } i=1,2,
\end{equation}
that is $\T_{n,y_i(m)}$ is the (sub-)tree consisting of all descendants of $(n,y_i(m))$ in $\T$. Thus, the union in \eqref{Tdec} is over disjoint sets. Note in particular that $\T_{n,y_i(m)} \in \Lambda_{n, y_i(m)}$. By construction, the subsets $K_{\T_{n,y_i(m)}}$ $\big(\subset \widetilde{B}_{n,y_i(m)} \big)$, for $i=1,2$, see \eqref{K_T}, satisfy $K_{\T_{n,y_1(m)}} \cap K_{\T_{n,y_2(m)}}= \emptyset$. For sake of clarity, and since $m$ and $\T$ will be fixed throughout the proof, we abbreviate 
\vspace{-0.3ex}
\begin{equation} \label{TiKi_abbr}
\T_{n,y_i(m)} = \T_i \quad \text{ and } \quad K_{\T_{n,y_i(m)}} = K_i, \quad \text{ for } i=1,2. 
\end{equation}

\vspace{-0.3ex}

In order to estimate the probability of the event $A_{\T}^h =A_{\T_1}^h \cap A_{\T_2}^h$, $h \in \mathbb{R}$, defined in \eqref{2.7}, we introduce a parameter $\alpha >0$ and write
\begin{align} \label{P2.2.1}
\P[A_{\T}^h] 
&\leq \P \big[ A_{\T_1}^h \cap A_{\T_2}^h \cap \big\{ \text{max}_{_{K_1}} \varphi \leq \alpha \big\} \big] \ + \ \P \big[ \text{max}_{_{K_1}} \varphi > \alpha  \big] \nonumber \\
& = \mathbb{E} \big[ 1_{A_{\T_1}^h} \cdot 1_{ \{ \text{max}_{_{K_1}} \varphi  \ \leq   \alpha \} } \cdot \mathbb{P} [A_{\T_2}^h \ | \varphi_{_{K_1}}] \big] \ + \ \P \big[  \text{max}_{_{K_1}} \varphi > \alpha  \big], 
\end{align} 
where $\text{max}_{_{K_1}} \varphi = \max \{ \varphi_x  ; \  x\in K_1 \}$ and the second line follows because $A_{\T_1}^h \cap \big\{ \text{max}_{_{K_1}} \varphi \leq \alpha \big\}$ is measurable with respect to $\sigma (\varphi_{_{K_1}})$, c.f. \eqref{Ameas}. We begin by focusing on the conditional probability $\mathbb{P} [A_{\T_2}^h \ | \varphi_{_{K_1}}]$ in \eqref{P2.2.1}. Using \eqref{phi_cond_exps} and \eqref{Ameas} applied to $A_{\T_2}^h$, and with a slight abuse of notation, we find
\begin{equation} \label{P2.2.2}
\mathbb{P} [A_{\T_2}^h \ | \varphi_{_{K_1}}] = \widetilde{ \mathbb{P} } \big[ A_{\T_2}^h \big( (\widetilde{\varphi}_x + \mu_x)_{x \in K_2} \big)\ \big], \qquad \P\text{-almost surely},
\end{equation}
where $ \mu_x = E_x \big[H_{K_1} < \infty , \varphi_{X_{H_{K_1}}} \big]$. On the event $\big\{ \text{max}_{_{K_1}} \varphi \leq \alpha \big\}$, we have, for all $x \in K_2$,
\vspace{-0.6ex}
\begin{equation} \label{P2.2.3}
\mu_x =  \sum_{y \in K_1} \varphi_y P_x [H_{K_1} < \infty, X_{H_{K_1}} = y ] \leq  \alpha \cdot P_x [H_{K_1} < \infty] \stackrel{\text{def.}}{=} m_x (\alpha),
\end{equation}

\vspace{-0.6ex}
\noindent which is deterministic and linear in $\alpha$. Moreover, we can bound $m_x(\alpha)$ as follows. By virtue of \eqref{1.13}, $P_x [H_{K_1} < \infty] \leq \text{cap}(K_1) \cdot \sup_{y \in K_1} g(x,y)$ for all $x \in K_2$. Since $K_1$ consists of $2^n$ disjoint boxes of side length $3L_0$, c.f. \eqref{TiKi_abbr} and \eqref{K_T}, its capacity can be bounded, using \eqref{1.12} and \eqref{1.14}, as $\text{cap}(K_1) \leq c 2^n L_0^{d-2}$. By \eqref{1.15}, \eqref{2.1} and the observation that $|x-y|\geq c' L_{n+1}$ whenever $x \in K_1$ and $y \in K_2$, it follows that 
\vspace{-0.5ex}
\begin{equation} \label{P2.2.4}
m_x (\alpha)  \leq c_1 \big( 2g(0)\big)^{-1/2} \cdot \alpha  \cdot 2^n l_0^{-(n+1)(d-2)} \ \stackrel{\text{def.}}{=} \frac{\gamma}{2}, \qquad \text{for } x \in K_2,
\end{equation}
which defines the constant $c_1$ from \eqref{2.10}, and the factor $( 2g(0))^{-1/2} $ is kept for later convenience.

Returning to the conditional probability $\mathbb{P} [A_{\T_2}^h \ | \varphi_{_{K_1}}]$, we first observe that, on the event $\big\{ \text{max}_{_{K_1}} \varphi \leq \alpha \big\}$ and for any $x \in K_2$, the inequality $\widetilde{\varphi}_x + \mu_x \geq h$ implies 
\vspace{-0.8ex}
\begin{equation*}
\widetilde{\varphi}_x -m_x(\alpha) \geq h - \mu_x -m_x(\alpha) \stackrel{\eqref{P2.2.3}}{\geq} h - 2m_x(\alpha) \stackrel{\eqref{P2.2.4}}{\geq} h - \gamma.
\end{equation*}

\vspace{-0.5ex}
\noindent Hence, on the event $\big\{ \text{max}_{_{K_1}} \varphi \leq \alpha \big\}$,
\vspace{-0.5ex}
\begin{equation} \label{P2.2.5}
\begin{split}
\begin{array}{lcl}
\mathbb{P} [A_{\T_2}^h \ | \varphi_{_{K_1}}]
&\stackrel{\eqref{P2.2.2}}{=} & \hspace{-1ex} \widetilde{ \mathbb{P} } \big[ A_{\T_2}^h \big( (\widetilde{\varphi}_x + \mu_x)_{x \in K_2} \big)\ \big] \\[1.2ex]
&\leq & \hspace{-1ex} \widetilde{ \mathbb{P} } \big[ A_{\T_2}^{h - \gamma} \big( (\widetilde{\varphi}_x -m_x(\alpha))_{x \in K_2} \big)\ \big] = \mathbb{P} [A_{\T_2}^{h - \gamma} \ | \varphi_{\vert_{K_1}} = - \alpha],
\end{array}
\end{split}
\end{equation}
where the last equality follows by \eqref{phi_cond_exps}, noting that, on the event $ \{ \varphi_{\vert_{K_1}} = - \alpha \}$, we have $\mu_x = m_x(-\alpha) = - m_x(\alpha)$ for all $x \in K_2$, c.f. \eqref{P2.2.3}. Applying Lemma \ref{FKG} to the right-hand side of \eqref{P2.2.5}, we immediately obtain that, on the event $\big\{ \text{max}_{_{K_1}} \varphi \leq \alpha \big\}$,
\vspace{-0.7ex}
\begin{equation} \label{P2.2.6}
\mathbb{P} [A_{\T_2}^h \ | \varphi_{K_1}] \leq \mathbb{P} [A_{\T_2}^{h - \gamma} \ | \varphi_{\vert_{K_1}} \geq - \alpha] \leq \mathbb{P} [A_{\T_2}^{h - \gamma} ] \cdot \big( \mathbb{P} [ \varphi_{\vert_{K_1}} \geq - \alpha] \big)^{-1}.
\end{equation}

\vspace{-0.2ex}
\noindent At last, we insert \eqref{P2.2.6} into \eqref{P2.2.1}, noting that, since $\varphi$ has the same law as $-\varphi$, we have $\mathbb{P} [ \varphi_{\vert_{K_1}} \geq - \alpha] = 1- \P[\text{min}_{_{K_1}} \varphi < - \alpha] = 1- \P[\text{max}_{_{K_1}} \varphi > \alpha] $, to get
\vspace{-1.0ex}
\begin{equation} \label{P2.2.7}
\P[A_{\T}^h] \leq \P \big[ A_{\T_1}^h \big] \cdot \P \big[ A_{\T_2}^{h - \gamma} \big] \cdot \big( 1- \P[\text{max}_{_{K_1}} \varphi > \alpha] \big)^{-1}  +  \P \big[  \text{max}_{_{K_1}} \varphi > \alpha  \big]. 
\end{equation}

Next, we turn our attention to the term $\P \big[  \text{max}_{_{K_1}} \varphi > \alpha  \big]$. By virtue of the BTIS-inequality (c.f \cite{AT}, Theorem 2.1.1), for arbitrary $ \emptyset \neq K \subset \subset \mathbb{Z}^d$, we have
\begin{equation} \label{P2.2.8}
\P \big[  \text{max}_{_{K}} \varphi > \alpha  \big] \leq  \exp \bigg\{-\frac{\big( \alpha - \E\big[  \text{max}_{_{K}} \varphi \big] \big)^2}{2 g(0)} \bigg\}, \qquad \text{if} \ \alpha > \E\big[  \text{max}_{_{K}} \varphi \big].
\end{equation}
In order to bound $\E\big[  \text{max}_{_{K}} \varphi \big]$, we write, using Fubini's theorem,
\begin{equation} \label{max_phi_1}
\E\big[  \text{max}_{_{K}} \varphi \big] \leq \E\big[  \text{max}_{_{K}} \varphi^+ \big] = \int_0^\infty  \text{d}u \  \P \big[  \text{max}_{_{K}} \varphi^+ > u \big] \leq A + \int_A^\infty  \text{d}u \  \P \big[  \text{max}_{_{K}} \varphi^+ > u \big],
\end{equation}
for arbitrary $A \geq 0$. Recalling that $\E[\varphi_x^2]= g(0)$ for all $x \in \mathbb{Z}^d$, c.f. \eqref{phi}, and introducing an auxiliary variable $\psi \sim \mathcal{N}(0,1)$, we can bound the integrand as
\begin{equation*}
P \big[  \text{max}_{_{K}} \varphi^+ > u \big] \leq |K| \cdot \P[\varphi_0 > u] =  |K| \cdot \P\big[\psi > g(0)^{-1/2}u \big] \leq |K| \cdot e^{-u^2/2g(0)},
\end{equation*}
where we have used in the last step that $\P [\psi > a] \leq e^{-a^2/2}$, for $a > 0$, which follows readily from Markov's inequality, since $\P [\psi > a] \leq \min_{\lambda > 0}e^{-\lambda a}\E[e^{\lambda \psi}] = \min_{\lambda > 0}e^{-\lambda a + \lambda^2/2}$, and the minimum is attained at $\lambda =a$. Inserting the bound for $P \big[  \text{max}_{_{K}} \varphi^+ > u \big]$ into \eqref{max_phi_1} yields, for arbitrary $A > 0$,
\begin{equation} \label{max_phi_2}
\E\big[  \text{max}_{_{K}} \varphi \big] \leq A +  |K| \int_A^\infty \text{d}u \ e^{-\frac{u^2}{2g(0)}} \leq A + c |K| \cdot e^{-A^2/2g(0)}.
\end{equation}
We select  $A =  (2 g(0) \log|K|)^{1/2}$ (so that $e^{-A^2/2g(0)}= |K|^{-1}$), by which means \eqref{max_phi_2} readily implies that
\begin{equation} \label{max_phi_final}
\E\big[  \text{max}_{_{K}} \varphi \big] \leq c \sqrt{\log|K|}, \qquad \text{for all} \quad \emptyset \neq K \subset \subset \mathbb{Z}^d.
\end{equation}
In the relevant case $K=K_1$ with $|K_1| = 2^n (3L_0)^d$, we thus obtain
\begin{equation} \label{P2.2.10}
\E\big[  \text{max}_{_{K_1}} \varphi \big] \leq c_2 \big( 2g(0) \log(2^n (3L_0)^d)\big)^{1/2} \stackrel{\eqref{2.8.1}}{=} \sqrt{2 g(0)} \cdot M(n, L_0),
\end{equation}
where the first inequality defines the constant $c_2$ from \eqref{2.8.1}. We now require
\begin{equation} \label{P2.2.11}
\alpha / \sqrt{2g(0)} \geq \sqrt{\log 2} + M(n, L_0),
\end{equation}
thus \eqref{P2.2.8} applies and yields
\begin{equation} \label{P2.2.12}
\P \big[  \text{max}_{_{K_1}} \varphi > \alpha  \big] \leq \min \Big\{ 1/2, \  e^{ - \big(\frac{\alpha}{\sqrt{2g(0)}} -M(n,L_0)\big)^2} \Big\}.
\end{equation}
Returning to \eqref{P2.2.7}, and using that $(1-x)^{-1} \leq 1 + 2x$ for all $0 \leq x \leq 1/2$ (with $x = \P \big[  \text{max}_{_{K_1}} \varphi > \alpha  \big] $), we finally obtain, for all $\alpha$ satisfying \eqref{P2.2.11} and $h' \geq h$,
\begin{equation} \label{P2.2.13}
\begin{split}
\begin{array}{lcl}
\P[A_{\T}^{h'}] \ \leq \ \P[A_{\T}^h] \hspace{-1ex} & \leq & \hspace{-1ex} \P \big[ A_{\T_1}^h \big] \cdot \P \big[ A_{\T_2}^{h - \gamma} \big]  + 3 \cdot \P \big[  \text{max}_{_{K_1}} \varphi > \alpha  \big] \\
&\stackrel{\eqref{P2.2.12}}{\leq} & \hspace{-1ex} \P \big[ A_{\T_1}^{h- \gamma} \big] \cdot \P \big[ A_{\T_2}^{h - \gamma} \big]  + 3 e^{-(\beta - M(n,L_0))^2},
\end{array}
\end{split}
\end{equation}
where we have set $\beta = \alpha / \sqrt{2g(0)}$. The claim \eqref{2.11} now readily follows upon taking suprema over all $\T \in \Lambda_{n+1,x}$ on both sides of \eqref{P2.2.13}, letting $\beta_n\stackrel{\text{def.}}{=} \beta$, $h_n \stackrel{\text{def.}}{=} h - \gamma \in \mathbb{R}$ ($h$ was arbitrary), $h_{n+1}\stackrel{\text{def.}}{=}h'$, so that requiring $h_{n+1} = h' \geq h = h_n + \gamma$, by virtue of \eqref{P2.2.4}, is nothing but \eqref{2.10}. Noting condition \eqref{P2.2.11} for $\beta_n =\beta$, we precisely recover \eqref{2.9}. This concludes the proof of Proposition \ref{P2.2}.
\end{proof}

\begin{remark} \label{R2.3} $\quad$
\smallskip

\noindent 1) The bound \eqref{max_phi_final}, which we have derived using an elementary argument, also follows from a more general (and stronger) estimate. One knows that (in fact, this holds for a large class  of Gaussian fields, c.f. \cite{AT}, Theorem 1.3.3), 
\begin{equation} \label{P2.2.9}
\E\big[  \text{sup}_{_{K}} \varphi \big] \leq C \int_0^{\frac{1}{2} \sup_{x,y, \in K} d(x,y)} \sqrt{\log\big( N(\varepsilon)\big)} \ \text{d} \varepsilon,
\end{equation}
where $d(x,y)= \big( \E[(\varphi_x - \varphi_y)^2] \big)^{1/2},$ $x,y \in \mathbb{Z}^d$, $K \subset \subset \mathbb{Z}^d$, $N(\varepsilon)$ denotes the smallest number of closed balls of radius $\varepsilon$ in this metric covering $K$, and $C$ is a universal constant. Clearly, $N(\varepsilon) \leq |K|$ for all $\varepsilon \geq 0$. Moreover, $\sup_{x,y, \in \mathbb{Z}^d} d(x,y) \leq \sqrt{2 g(0)}$ by virtue of \eqref{phi}. Inserting this into \eqref{P2.2.9} immediately yields the bound \eqref{max_phi_final}.
\vspace{0.3cm}

\noindent 2) We mention a generalization of Proposition \ref{P2.2}, which is of independent interest, but will not be needed in what follows. Consider integers $L_0 \geq 1$, $l_0 \geq 100$, and a collection $D_x$, $x \in \IL_0$, of events in $\Omega$ ($=\{ 0,1\}^{\mathbb{Z}^d}$, see above \eqref{1.18}), such that
\begin{equation} \label{di1}
D_x \text{ is } \sigma(Y_z \ ; \ z \in \widetilde{B}_{0,x})\text{-measurable for each $x \in \IL_0$},
\end{equation}
where $Y_z$, $z \in \mathbb{Z}^d$, stand for the canonical coordinates on $\Omega$.

Given $h \in \mathbb{R}$, $n \geq 0$, $x \in \IL_n$, and $\T \in \Lambda_{n,x}$, we replace $A_{\T}^h$ in \eqref{2.7} by (see below \eqref{1.18} for the notation $\Phi^h$)
\begin{equation} \label{di2}
D_{\T}^h = \bigcap_{(0,y) \in \T \cap \I_0} (\Phi^h)^{-1} (D_y),
\end{equation}

\vspace{-0.8ex}
\noindent and $p_n(h)$ in \eqref{2.8} by
\begin{equation} \label{di3}
q_n(h) = \sup_{x \in \IL_n, \T \in \Lambda_{n,x}} \P[D_{\T}^h].
\end{equation}
One then has the following generalization of Proposition \ref{P2.2}:

\begin{propbis} $(L_0 \geq 1$, $l_0 \geq 100$, \eqref{di1}$)$
\smallskip

\noindent Assume that
\vspace{-0.5ex}
\begin{equation} \label{di4}
\text{for each $x \in \IL_0$, $D_x$ is increasing,}
\end{equation}
that $(\beta_n)_{n \geq 0}$ is a positive sequence, $(h_n)_{n \geq 0}$ a real increasing sequence, such that \eqref{2.9}, \eqref{2.10} hold. Then,
\vspace{-0.5ex}
\begin{equation} \label{di5}
q_{n+1}(h_{n+1}) \leq q_n(h_n)^2 + 3  e^{-(\beta_n -M(n,L_0))^2}, \qquad \text{for all } n \geq 0.
\end{equation}
If instead, 
\vspace{-0.5ex}
\begin{equation} \label{di6}
\text{for each $x \in \IL_0$, $D_x$ is decreasing,}
\end{equation}
$(\beta_n)_{n \geq 0}$ is a positive sequence, $(h_n)_{n \geq 0}$ a real decreasing sequence, so that \eqref{2.9} holds and \eqref{2.10} holds for $(-h_n)_{n \geq 0}$, then \eqref{di5} holds as well.
\end{propbis}

\begin{proof}
The arguments employed in the proof of Proposition \ref{P2.2} yield the first statement (with $D_x$, $x \in \IL_0$, increasing events). To derive the second statement (when $D_x$, $x \in \IL_0$, are decreasing events), one argues as follows. One introduces the inversion $\iota: \Omega \longrightarrow \Omega$ such that $Y_z \circ \iota = 1 - Y_z$, for all $z \in \mathbb{Z}^d$, and the collection of ``flipped'' events $\overline{D}_x = \iota^{-1}(D_x)= \iota(D_x)$, $x \in \IL_0$. One defines $\overline{D}_{\T}^h$ as in \eqref{di2} with $\overline{D}_y$, $y \in \IL_0$, in place of  $D_y$, $y \in \IL_0$. Now observe that $(-\varphi_x)_{x \in \mathbb{Z}^d}$ has the same law as $(\varphi_x)_{x \in \mathbb{Z}^d}$ under $\P$, and that for any $h \in \mathbb{R}$, $\big( 1\{ \varphi_x < -h \} \big)_{x \in \mathbb{Z}^d}$ has the same law $Q^h$ as $\big( 1\{ \varphi_x \geq h \} \big)_{x \in \mathbb{Z}^d}$ under $\P$. From this, we infer that for all $h \in \mathbb{R}$, $x \in \IL_n$, and $\T \in \Lambda_{n,x}$,
\begin{equation} \label{di7}
\begin{split}
\P \big[D_{\T}^h \big] &= Q^h \Big[ \bigcap_{(0,y) \in \T \cap \I_0} D_y \Big] = \P \Big[ \big( 1\{ \varphi_x < -h \} \big)_{x \in \mathbb{Z}^d} \in \bigcap_{(0,y) \in \T \cap \I_0} D_y \Big] \\
&= Q^{-h} \Big[ \bigcap_{(0,y) \in \T \cap \I_0} \overline{D}_y \Big] = \P \big[\ \overline{D}_{\T}^{-h} \big]. 
\end{split}
\end{equation}
When \eqref{di6} holds, the events $\overline{D}_x$, $x \in \IL_0$, satisfy \eqref{di4}, and thanks to the identity \eqref{di7}, the second statement of Proposition 2.2' is reduced to the first statement. This concludes the proof of Proposition 2.2'.
\end{proof}

\noindent 3) There is an analogy between Proposition 2.2' and the main renormalization step Theorem 2.1 of \cite{S3}, for the decoupling inequalities of random interlacements (see Theorem 2.6 of \cite{S3}). Note however, that unlike condition (2.7) of \cite{S3} (see also (2.70) in \cite{S3}), \eqref{2.9} and \eqref{2.10} tie in the finest scale $L_0$ to the sequence $(h_n)_{n \geq 0}$. This feature has to do with the role of the cut-off level $\alpha$ we introduce in \eqref{P2.2.1} and the remainder term it produces. \hfill $\square$
\end{remark}

We now return to Proposition \ref{P2.2} and aim at propagating the estimate \eqref{2.11} inductively. To this end, we first define, for all $n \geq 0$,
\begin{equation} \label{2.12}
\beta_n = (\log 2)^{1/2} + M(n,L_0) + 2^{(n+1)/2} \big(n^{1/2} + K_0^{1/2}\big),
\end{equation}
where $K_0 > 0$ is a certain parameter to be specified below in Proposition \ref{P2.5} and later in \eqref{K0}. Note in particular that condition \eqref{2.9} holds for this choice of $(\beta_n)_{n\geq0}$.  In the next proposition, we inductively derive bounds for $p_n(h_n)$, $n \geq 0$, given any sequence $(h_n)_{n\geq 0}$ satisfying the assumptions of Proposition \ref{P2.2}, provided the induction can be initiated, see \eqref{2.13} below.
 
\begin{proposition} \label{P2.5} $\quad$
\smallskip

\noindent Assume $h_0 \in \mathbb{R}$ and $K_0 \geq 3(1-e^{-1})^{-1} \stackrel{\textrm{def.}}{=}B$ are such that
\begin{equation} \label{2.13}
p_0(h_0) \leq e^{-K_0},
\end{equation}
and let  the sequence $(h_n)_{n \geq 0}$ satisfy \eqref{2.10} with  $(\beta_n)_{n\geq0}$ as defined in \eqref{2.12}. Then,
\begin{equation} \label{2.14}
p_n(h_n) \leq e^{-(K_0 - B)2^n}, \qquad \text{for all } n \geq 0. 
\end{equation}
\end{proposition}

\begin{proof}
We define a sequence $(K_n)_{n \geq 0}$ inductively by
\begin{equation} \label{P2.5.1}
K_{n+1} = K_n - \log \Big( 1+ e^{K_n} \cdot 3^{2^{-(n+1)}} e^{-2^{-(n+1)}(\beta_n - M(n,L_0))^2} \Big), \qquad \text{ for all } n \geq 0,
\end{equation}
with $\beta_n$ given by \eqref{2.12} (the factor following $e^{K_n}$ in \eqref{P2.5.1} should be viewed as the $2^{(n+1)}$-th root of the remainder term on the right-hand side of \eqref{2.11}). Then, \eqref{P2.5.1} implies that $K_n \leq K_0$ for all $n \geq 0$. Moreover, as we now see,
\begin{equation} \label{P2.5.2}
K_n \geq K_0 - B, \qquad \text{for all } n \geq 0.
\end{equation}
This is clear for $n =0$. When $n \geq 1$, first note that by virtue of \eqref{P2.5.1}, 
\begin{equation} \label{P2.5.3}
K_n = K_0 - \sum_{m=0}^{n-1} \log \Big( 1+ e^{K_m} \cdot 3^{2^{-(m+1)}} e^{-2^{-(m+1)}(\beta_m - M(m,L_0))^2} \Big), \qquad \text{for all } n \geq 1.
\end{equation}
Moreover, \eqref{2.12} implies 
\begin{equation} \label{P2.5.4}
(\beta_m - M(m,L_0))^2  \geq \log 2 + 2^{m+1}(m^{1/2} + K_0^{1/2})^2 \geq \log 2 + 2^{m+1}(m + K_0),
\end{equation}
for all $m \geq 0$, which, inserted into \eqref{P2.5.3}, yields
\begin{equation*} 
K_n  \geq K_0 -  \sum_{m=0}^{\infty} \log \Big( 1+ e^{K_m} \cdot 3^{2^{-(m+1)}} e^{-K_0 -m} \Big) \geq K_0 - 3 \sum_{m=0}^\infty e^{-m} = K_0 -B,
\end{equation*}
where we have used $K_n \leq K_0$ and $\log(1+x) \leq x$ for all $x \geq 0$ in the second inequality. Hence, \eqref{P2.5.2} holds. We will now show by induction on $n$ that
\begin{equation} \label{P2.5.5}
p_n (h_n) \leq e^{-K_n 2^n}, \qquad \text{for all } n \geq 0,
\end{equation}
which, together with \eqref{P2.5.2}, implies \eqref{2.14}. The inequality \eqref{P2.5.5} holds for $n=0$ by assumption, c.f. \eqref{2.13}. Assume now it holds for some $n$. By Proposition \ref{2.2}, we find
\begin{align*}
p_{n+1}(h_{n+1}) 
&\stackrel{\eqref{2.11}}{\leq}  \big( e^{-K_n 2^n} \big)^2 +    3 e^{-(\beta_n - M(n,L_0))^2}  \\
&\leq \Big[ e^{-K_n} \big( 1+ e^{K_n}3^{2^{-(n+1)}} e^{-2^{-(n+1)}(\beta_n - M(n,L_0))^2} \big) \Big]^{2^{n+1}} \stackrel{\eqref{P2.5.1}}{=} e^{-K_{n+1}2^{n+1}}.
\end{align*}
This concludes the proof of \eqref{P2.5.5} and thus of Proposition \ref{P2.5}.
\end{proof}

\begin{remark} $\quad$
\medskip

\noindent Although we will not need this fact in what follows, let us point out that a straightforward adaptation of Proposition \ref{P2.5} holds in the context of Proposition 2.2'. \hfill $\square$
\end{remark}

\bigskip
We will now state the main theorem of this section and prove it using Proposition \ref{P2.5}. To this end, we select $K_0$ appearing in Proposition \ref{P2.5} as follows:
\begin{equation} \label{K0}
K_0 = \log(2c_0 l_0^{2(d-1)}) + B \qquad \text{(see \eqref{2.6.1} for the definition of $c_0$)}.
\end{equation}
Moreover, we will solely consider sequences $(h_n)_{n \geq0}$ with
\begin{equation} \label{h_n}
h_0 > 0,  \qquad h_{n+1} - h_n = c_1 \beta_n \big(2 l_0^{-(d-2)} \big)^{n+1}, \qquad \text{for all } n \geq 0,
\end{equation}
so that condition \eqref{2.10} is satisfied. We recall that $\beta_n$ is given by \eqref{2.12}, which now reads
\begin{equation} \label{beta_n}
\beta_n = (\log 2)^{1/2} + c_2 \big( \log(2^n(3L_0)^d)\big)^{1/2} + 2^{(n+1)/2} \big( n^{1/2} + ( \log(2c_0 l_0^{2(d-1)}) + B)^{1/2} \big)
\end{equation}
where we have substituted $M(n,L_0)$ from \eqref{2.8.1} and $K_0$ from \eqref{K0}. Note that $L_0,$ $l_0$ and $h_0$ are the only parameters which remain to be selected. We finally proceed to the main

\begin{theorem} \label{T2.6}$\quad$
\smallskip

\noindent The critical point $h_{**}(d)$ defined in \eqref{h**} satisfies
\begin{equation} \label{2.20}
h_{**}(d) < \infty, \qquad \text{for all } d \geq 3.
\end{equation}
Moreover, for all $d \geq 3$ and $h > h_{**}(d)$, there exist positive constants $c(h)$, $c'(h)$ and $0< \rho <1$ $(\rho$ depending on $d$ and $h)$ such that
\begin{equation} \label{2.21}
\P \big[ B(0,L) \stackrel{\geq h}{\longleftrightarrow}  S(0,2L) \big] \leq c(h) \cdot e^{-c'(h)L^{\rho}}, \qquad \text{for all } L \geq 1.
\end{equation} 
In particular, the connectivity function $\P \big[ 0 \stackrel{\geq h}{\longleftrightarrow}  x \big]$ of the excursion set above level $h$ has stretched exponential decay, i.e. there exists $c''(h)>0$ such that 
\begin{equation} \label{2.21bis}
\P \big[ 0 \stackrel{\geq h}{\longleftrightarrow}  x \big] \leq c(h) \cdot e^{-c''(h)|x|^{\rho}}, \qquad \text{for all } x \in \mathbb{Z}^d, \ h > h_{**}(d), \text{ and } d \geq 3.
\end{equation} 
\end{theorem}

\begin{cor} \label{C2.7} $\quad$
\smallskip

\noindent The excursion set $E_\varphi^{\geq h}$ above level $h$ defined in \eqref{Ephi} undergoes a non-trivial percolation phase transition for all $d \geq 3$, i.e. 
\begin{equation} \label{2.22}
 (0 \leq) \ h_*(d) < \infty, \qquad \text{for all } d \geq 3, 
\end{equation}
and
\begin{equation} \label{2.23}
\P[E_{\varphi}^{\geq h} \text{ contains an infinite cluster }] = \left \{
\begin{array}{rl}
 1, & \text{ if  } h < h_* \\
 0, & \text{ if  } h > h_*.
\end{array}
\right.
\end{equation}
\end{cor}

\begin{proof6}
The lower bound $h_*(d) \geq 0$ in \eqref{2.22} follows from Corollary 2 of \cite{BLM}. In order to establish the finiteness in \eqref{2.22}, it suffices to show $h_* \leq h_{**}$ and to invoke the above Theorem \ref{T2.6}. To this end, we note that by definition (c.f. \eqref{eta}),
\begin{equation} \label{C2.7.1}
\eta (h) \leq \P \big[ B(0,L) \stackrel{\geq h}{\longleftrightarrow} S(0,2L) \big], \qquad \text{for all } L \geq 1, 
\end{equation}
and $h_* \leq h_{**}$ readily follows. As for \eqref{2.23}, it is an immediate consequence of Lemma \ref{perc_prop}. \hfill $\square$
\end{proof6}

\medskip

\begin{proof1}
To prove \eqref{2.20}, it suffices to construct an explicit level $\bar{h}$ with $0 < \bar{h} < \infty$ such that $\P \big[ B(0,L) \stackrel{\geq \bar{h}}{\longleftrightarrow}  S(0,2L) \big]$ decays polynomially in $L$, as $L \to \infty$. In fact, we will even show that $\P \big[ B(0,L) \stackrel{\geq \bar{h}}{\longleftrightarrow}  S(0,2L) \big]$ has stretched exponential decay. 

We begin by observing that the sequence $(h_n)_{n \geq 0}$ defined in \eqref{h_n} has a finite limit $h_\infty = \lim_{n \to \infty} h_n$ for every choice of $L_0,$ $l_0$ and $h_0$. Indeed, $\beta_n$ as given by \eqref{beta_n} satisfies $\beta_n \leq c(L_0,l_0)2^{n+1}$ for all $n \geq 0$, hence 
\begin{equation*}
h_\infty \stackrel{\eqref{h_n}}{=} h_0 +  c_1 \sum_{n=0}^\infty \beta_n \big(2 l_0^{-(d-2)} \big)^{n+1} \leq h_0 + c'(L_0,l_0) \sum _{n=0}^\infty \big(4 l_0^{-(d-2)} \big)^{n+1} < \infty,
\end{equation*}
since we assumed $l_0 \geq 100$. We set
\begin{equation} \label{L0l0}
L_0 = 10, \qquad l_0 =100,
\end{equation}
and now show with Proposition \ref{P2.5} that there exists $h_0 >0$ sufficiently large such that, defining
\begin{equation} \label{barh}
\bar{h} = h_\infty = \lim_{n \to \infty} h_n \quad (< \infty), 
\end{equation}
we have
\begin{equation} \label{GOAL1}
\P \big[ B(0,L) \stackrel{\geq \bar{h}}{\longleftrightarrow}  S(0,2L) \big] \leq c \cdot e^{-c'L^\rho}, \qquad \text{ for all } L \geq 1,
\end{equation}
for suitable $c,c' >0$ and $0 < \rho <1$. To this end, we note that $p_0 (h_0)$ defined in \eqref{p0} is bounded \nolinebreak by
\begin{align*}
p_0 (h_0) \stackrel{\eqref{p0}}{\leq} \P \big[\text{max}_{_{\widetilde{B}_{0,x=0}}} \varphi \geq h_0 \big] \stackrel{\eqref{P2.2.8}}{\leq} \exp \bigg\{-\frac{\big( h_0 - \E\big[  \text{max}_{_{\widetilde{B}_{0,x=0}}} \varphi \big] \big)^2}{2 g(0)} \bigg\}, 
\end{align*}
when $h_0 \geq c$ (e.g. using \eqref{max_phi_final} to bound $\E\big[  \text{max}_{_{\widetilde{B}_{0,x=0}}} \varphi \big]$). In particular, since $K_0$ in \eqref{K0} is completely determined by the choices \eqref{L0l0}, we see that $p_0(h_0) \leq e^{-K_0}$ for all $h_0 \geq c$, i.e. condition \eqref{2.13} holds for sufficiently large $h_0$. By Proposition \ref{P2.5}, setting $h_0 = c$ and $\bar{h}$ as in \eqref{barh}, we obtain
\begin{equation} \label{T2.6.3}
p_n(\bar{h}) \stackrel{\bar{h} > h_n}{\leq} p_n(h_n) \stackrel{\eqref{2.14}}{\leq} e^{-(K_0 - B)2^n} \stackrel{\eqref{K0}}{=} \big( 2c_0 l_0^{2(d-1)} \big)^{-2^n} , \qquad \text{for all } n \geq 0.
\end{equation}
Therefore, we find that for all $n \geq 0$ and $x \in \IL_n$,
\begin{equation} \label{T2.6.4}
\P \big[ B_{n,x} \stackrel{\geq \bar{h}}{\longleftrightarrow}  \partial^i \widetilde{B}_{n,x} \big] \stackrel{\eqref{2.7}}{\leq } |\Lambda_{n,x}| \cdot p_n (\bar{h}) \stackrel{\eqref{2.6.1}, \eqref{T2.6.3}}{\leq} \big( c_0 l_0^{2(d-1)} \big)^{2^n} \big( 2c_0 l_0^{2(d-1)} \big)^{-2^n} = 2^{-2^n}.
\end{equation}
We now set $\rho = \log 2 / \log l_0$, whence $2^n = l_0^{n\rho} = (L_n / L_0)^{\rho}$. Given $L \geq 1$, we first assume there exists $n \geq 0$ such that $2L_n \leq L < 2L_{n+1}$. Then, since 
\begin{equation*}
\P \big[ B(0,L) \stackrel{\geq \bar{h}}{\longleftrightarrow}  S(0,2L) \big] \leq \P \bigg[ \bigcup_{x \in \IL_n : B_{n,x} \cap  S(0,L) \neq \emptyset} \big\{ B_{n,x} \stackrel{\geq \bar{h}}{\longleftrightarrow}  \partial^i \widetilde{B}_{n,x} \big\} \bigg],
\end{equation*}
and the number of sets contributing to the union on the right-hand side is bounded by $cl_0^{d-1}$, \eqref{T2.6.4} readily implies \eqref{GOAL1}, and by adjusting $c,c'$, \eqref{GOAL1} will hold for $L < 2 L_0$ as well. It follows that $\bar{h} \geq h_{**}$, which completes the proof of \eqref{2.20}.
 
We now turn to the proof of \eqref{2.21}. Let $h$ be some level with $h_{**} < h < \infty$, and define $h_0 = (h_{**} + h)/2$. Since $h_0 > h_{**}$, we may choose $\varepsilon = \varepsilon(h) > 0$ such that $\lim_{L \to \infty} L^\varepsilon  \P  \big[ B(0,L) \stackrel{\geq h_0}{\longleftrightarrow}  S(0,2L) \big] = 0$, which readily implies 
\begin{equation} \label{T2.6.5}
\lim_{L_0 \to \infty } L_0^\varepsilon \cdot p_0(h_0) = 0, 
\end{equation}
(see \eqref{p0} for the definition of $p_0(\cdot)$). Moreover, we let
\begin{equation} \label{l0}
l_0 = 100 \big( \big[ L_0^{\frac{\varepsilon}{3(d-1)}} \big] +1 \big),
\end{equation}
so that $l_0 \geq 100$ is an integer, as required. From \eqref{beta_n}, it is then easy to see, using \eqref{l0}, that $\beta_n \leq c(h) \log(l_0) 2^{n+1}$, for all $n \geq 0$. Hence, the limit $h_\infty= \lim_{n \to \infty} h_n$ of the increasing sequence $(h_n)_{n \geq 0}$ defined in \eqref{h_n} satisfies
\begin{equation} \label{2.19}
\begin{split}
h_\infty = h_0 + c_1 \sum_{n=0}^\infty \beta_n \big(2 l_0^{-(d-2)} \big)^{n+1}
&\leq h_0 + c'(h) \log(l_0)l_0^{-(d-2)} \sum_{n=0}^\infty  \big(4 l_0^{-(d-2)} \big)^{n} \\
&= h_0 + c'(h)  \frac{\log(l_0)}{l_0^{d-2}} \cdot  \frac{1}{1- 4 l_0^{-(d-2)}} .
\end{split}
\end{equation}
Thus, \eqref{l0} and \eqref{2.19} imply that $h_\infty \leq h$ whenever $L_0 \geq c(h)$. Moreover, 
\begin{equation} \label{T2.6.6}
e^{-K_0} \stackrel{\eqref{K0}}{=} cl_0^{-2(d-1)} \stackrel{\eqref{l0}}{\geq}c' L_0^{-\frac{2 \varepsilon}{3}} \geq p_0(h_0), \qquad \text{for all } L_0 \geq c(h),
\end{equation}
where the last inequality follows by \eqref{T2.6.5}. We thus select $L_0 = c(h)$ so that both \eqref{T2.6.6} and $h_\infty \leq h$ hold. Since condition \eqref{2.13} is satisfied, Proposition \ref{P2.5} yields
\begin{equation*}
p_n(h) \stackrel{h \geq h_n}{\leq}p_n(h_n) \stackrel{\eqref{2.14}}{\leq} e^{-(K_0 - B)2^n} \stackrel{\eqref{K0}}{=} \big( 2c_0 l_0^{2(d-1)} \big)^{-2^n} , \qquad \text{for all } n \geq 0,
\end{equation*}
from which point on one may argue in the same manner as for the proof of \eqref{2.20} to infer \eqref{T2.6.4} (with $h$ in place of $\bar{h}$)  and subsequently deduce \eqref{2.21}. In particular, this involves defining $\rho = \log 2 / \log l_0$, which depends on $h$ (and $d$) through $l_0$. The stretched exponential bound \eqref{2.21bis} for the connectivity function of $E_\varphi^{\geq h}$ is an immediate corollary of \eqref{2.21}, since $\P\big[ 0 \stackrel{\geq h}{\longleftrightarrow}  x \big] \leq \P \big[ B(0,L) \stackrel{\geq h}{\longleftrightarrow}  S(0,2L) \big] \leq c(h) e^{- c''(h)|x|^\rho}$ whenever $ 2L \leq |x|_\infty < 2(L+1)$. \hfill $\square$
\end{proof1}

\begin{remark} \label{2.8} $\quad$
\smallskip

\noindent 1) An important open question is whether $h_*$ equals $h_{**}$ or not. In case the two differ, the decay of $\P \big[ 0 \stackrel{\geq h}{\longleftrightarrow} S(0,L) \big]$ as $L \to \infty$, for $h > h_*$, exhibits a sharp transition. Indeed, first note that by \eqref{h*}, for all $h > h_*$, $\P \big[ 0  \stackrel{\geq h}{\longleftrightarrow} S(0,L) \big] \longrightarrow 0$, as $ L \to \infty$. If $h_{**} > h_*$, then by definition of $h_{**}$,
\begin{equation*}
\text{for } h \in ( h_* , h_{**}) \text{ and any } \alpha > 0,  \qquad \limsup_{L \to \infty} L^{d-1+ \alpha} \P \big[ 0  \stackrel{\geq h}{\longleftrightarrow} S(0,L) \big] = \infty.
\end{equation*}
Hence $\P \big[ 0  \stackrel{\geq h}{\longleftrightarrow} S(0,L) \big]$ decays to zero with $L$, but with an at most polynomial decay for $h \in ( h_* , h_{**})$. However, for $h > h_{**}$, $\P \big[ 0  \stackrel{\geq h}{\longleftrightarrow} S(0,L) \big]$ has a stretched exponential decay in $L$, by \eqref{2.21}. 

\vspace{0.3cm}

\noindent 2) The proof of Theorem \ref{T2.6} works just as well if we replace the assumption $h> h_{**}$ by $h>\tilde{h}_{**}$, where $\tilde{h}_{**}(\leq h_{**})$ is defined similarly as $h_{**}$ in \eqref{h**}, simply replacing the ``lim'' by a ``liminf'' in \eqref{h**}, i.e.
\begin{equation}
\tilde{h}_{**} = \inf \big\{ h \in \mathbb{R} \; ; \text{for some $\alpha >0$, } \liminf_{L\to\infty} L^{\alpha } \; \P \big[ B(0,L) \stackrel{\geq h}{\longleftrightarrow} S(0, 2L) \big] = 0 \big\}.
\end{equation}
Hence $ \P \big[ B(0,L) \stackrel{\geq h}{\longleftrightarrow} S(0, 2L) \big]$ has stretched exponential decay in $L$ when $h > \tilde{h}_{**}$, and one has in fact the equality
\begin{equation}
h_{**} = \tilde{h}_{**}.
\end{equation}
\hfill $\square$
\end{remark}

\section{Positivity of $h_*$ in high dimension} \label{POSLARGEd}

The main goal of this section is the proof of Theorem \ref{T3.3} below, which roughly states that in high dimension, for small but \textit{positive} $h$, the excursion set $E_\varphi^{\geq h}$ contains an infinite cluster with probability $1$. We will prove the stronger statement that percolation already occurs in a two-dimensional slab $\mathbb{Z}^2 \times [0,2L_0) \times \{ 0\}^{d-3} \subset \mathbb{Z}^d$ for sufficiently large $L_0$, see \eqref{3.SlabPerc} below. 

%(here and in what follows, $[0,2L_0)$ stands for $[0,2L_0) \cap \mathbb{Z}$).

The proof essentially relies on two main ingredients. The first ingredient is a suitable decomposition, for large $d$, of the free field $\varphi$ restricted to $ \mathbb{Z}^3$ (viewed as a subset of $\mathbb{Z}^d$), into the sum of two independent Gaussian fields $\psi$ and $\xi$ (c.f. \eqref{3.3} and \eqref{3.4} below for their precise definition). The field $\psi$ is i.i.d. and the dominant part, while $\xi$ only acts as a ``perturbation.'' The key step towards this decomposition appears in Lemma \ref{L3.1}. 

The second ingredient is a Peierls-type argument, which comprises several steps: first, the sublattice $\M$ is partitioned into blocks of side length $L_0$, which are declared ``good'' if certain events defined separately for $\xi$ and $\psi$ occur simultaneously. Roughly speaking, these events are chosen in a way that suitable excursion sets of the dominant field $\psi$ percolate well and the perturbative part $\xi$ doesn't spoil this percolation (see \eqref{T3.3.6}, \eqref{T3.3.4} and \eqref{BAD} for precise definitions). Moreover, $*$-connected components of bad blocks are shown to have small probability, see Lemma \ref{PeierlsBound}. This ensures that the usual method of Peierls contours is applicable, which in turn allows the conclusion that an infinite cluster of $E_\varphi^{\geq h}$ exists within the above-mentioned slab with positive probability (and with probability $1$ by ergodicity). We note that $\xi$ doesn't have finite-range dependence, which renders impractical the use of certain well-known stochastic domination theorems (see for example \cite{Scho}, \cite{Pis}). 

One word on notation: in what follows, we identify $\mathbb{Z}^{k}$, $k=2,3$, with the set of points $(x^1, \dots, x^d ) \in \mathbb{Z}^d$ satisfying $x^{k+1} = x^{k+2} = \dots = x^d =0$. We recall that in this section, constants are numerical unless dependence on additional parameters is explicitly indicated. Moreover, we shall assume throughout this section that
 \begin{equation} \label{3.0}
 d \geq 6.
 \end{equation} 
We also recall that $g(\cdot,\cdot)$, c.f. \eqref{GreenFunction}, stands for the Green function on $\mathbb{Z}^d$. Without further ado, we begin with

\begin{lemma} \label{L3.1} $($Covariance decomposition$)$
\smallskip

\noindent Let $K= \mathbb{Z}^3$. There exists a function $g'$ on $K \times K$ such that
\begin{equation} \label{3.1}
g(x,y)= \sigma^2 (d) \cdot \delta(x,y) \ + \ g'(x,y), \qquad \text{ for all } x,y \in K,
\end{equation}
where $ 1/2 \leq \sigma^2 (d) < 1$, $\sigma^2(d) \to 1$ as $d \to \infty$, $\delta(\cdot , \cdot)$ denotes the Kronecker symbol, and $g'$ is the kernel of a translation invariant, bounded, positive operator $G'$ on $\ell^2 (K)$, which is the operator of convolution with $g'(\cdot,0)$. Its spectral radius $\rho(G')$ satisfies
 \begin{equation} \label{3.2}
 \rho(G') \leq c_3/d.
 \end{equation}
\end{lemma}

\begin{proof}
The operator $Af(x)= \sum_{y \in K}g(x,y)f(y)$, for $x \in K$ and $f \in \ell^2(K)$, is a convolution operator, which is bounded and self-adjoint on $ \ell^2(K)$, by \eqref{1.15}, \eqref{3.0}, as well as the translation invariance and the symmetry of $g(\cdot, \cdot)$ (letting $h(\cdot)=g(\cdot,0)$, we also use that $ \norm{Af}_{\ell^2 (K)} = \norm {h*f}_{\ell^2 (K)} \leq \norm{h}_{\ell^1 (K)} \norm{f}_{\ell^2 (K)}$, a special case of Young's inequality, see \cite{ReSi}, pp. 28-29). Moreover, by \cite{Sp}, P25.2 (b), p. 292, it has an inverse
\begin{equation} \label{L3.1.1}
A^{-1} = I - \Pi,
\end{equation}
where $\Pi$ is the bounded self-adjoint operator on $\ell^2(K)$, $\Pi f(x)= \sum_{y \in K}\pi(x,y)f(y)$, for $x \in K$ and $f \in \ell^2(K)$, with kernel
\begin{equation} \label{PI}
\pi(x,y) = P_x[\widetilde{H}_K < \infty, X_{\widetilde{H}_K} = y] = \pi (0,y-x), \qquad \text{for }x, y \in K. 
\end{equation}
Introducing
\begin{equation} \label{kappa}
\kappa = P_0[\widetilde{H}_K = \infty] \in (0,1) \qquad \text{(recall \eqref{3.0})},
\end{equation}
we can write
\begin{equation} \label{Gamma}
A^{-1} = \kappa I + \big( (1-\kappa) I - \Pi \big) \stackrel{\text{def.}}{=}\kappa I + \Gamma,
\end{equation}
where $\Gamma$ is the bounded self-adjoint operator on $\ell^2(K)$ defined by $\Gamma f(x)= \sum_{y \in K}\gamma(x,y)f(y)$, for $x \in K$, $f \in \ell^2(K)$, and
\begin{equation} \label{gamma}
\gamma(x,y) = \left \{
\begin{array}{ll}
- \pi(x,y), & \text{if } y \neq x \\
1- \kappa - \pi(x,x) \stackrel{\eqref{PI}, \eqref{kappa}}{=} P_x [\widetilde{H}_K < \infty, X_{\widetilde{H}_K} \neq x] = \sum_{y \neq x} \pi(x,y), & \text{if } y=x. 
\end{array}
\right.
\end{equation}
Note that by \eqref{PI}, \eqref{gamma}, $\gamma (x,y)=\gamma(0,y-x)$ and $\Gamma$ is a convolution operator on $\ell^2(K)$. By Young's inequality (see above \eqref{L3.1.1}), its operator norm $\norm{\Gamma}$ thus satisfies
\begin{equation} \label{L3.1.4}
\norm{ \Gamma } \leq \norm{\gamma(0,\cdot)}_{\ell^1(K)} =  \sum_{ y \neq 0} \pi(0,y) + \sum_{ y \neq 0} \pi(0,y) \leq 2 P_0\big[ \widetilde{H}_K < \infty \big]  \stackrel{\eqref{kappa}}{=} 2(1-\kappa).
\end{equation} 
Observe also that $\Gamma$ is a positive operator. Indeed, $(\Gamma f, f)_{\ell^2(K)} \stackrel{\eqref{gamma}}{=} \frac{1}{2} \sum_{x,y \in K} \pi(x,y)(f(x)- f(y))^2$, for all $f \in \ell^2(K)$, where $(\cdot, \cdot)_{\ell^2(K)}$ denotes the inner product in $\ell^2(K)$. By \eqref{L3.1.1} and \eqref{Gamma}, we can write, for arbitrary, $a \in (0,1)$, 
\begin{equation} \label{L3.1.3}
\begin{split}
A= \big( \kappa I + \Gamma \big)^{-1} &=  \kappa^{-1} \big( \kappa I + \Gamma - \Gamma \big)  \big( \kappa I +\Gamma \big)^{-1} \\
&=  \kappa^{-1} \big[ I -  \Gamma \big( \kappa I + \Gamma \big)^{-1} \big] \\
&= \kappa^{-1} \big[(1-a)I +T_a \big],
\end{split}
\end{equation}
where $T_a$ is the bounded operator on $\ell^2(K)$,
\begin{equation} \label{Ta}
T_a = aI - \Gamma \big( \kappa I + \Gamma \big)^{-1},
\end{equation}
which is self-adjoint by \eqref{L3.1.3}, since $A$ is. We will now select $a \in (0,1)$ in such a way that the operator $T_a$ is positive. By self-adjointness, we know that the spectrum $\sigma(T_{a})$ of $T_{a}$ satisfies
\begin{equation*}
\sigma(T_{a}) \subset [m(T_{a}), \rho(T_{a})] \subseteq \mathbb{R},
\end{equation*}
where $m(T_{a}) = \inf \{ ( T_{a}  f,f )_{\ell^2(K)} ; \norm{f}_{\ell^2(K)} = 1 \}$ and $ \rho (T_{a}) = \sup \{ ( T_{a}  f,f )_{\ell^2(K)} ; \norm{f}_{\ell^2(K)} = 1 \}$. By the spectral theorem, and using that the function $x \mapsto x/ (\kappa + x)$ is increasing in $x \geq 0$, it follows that
\begin{equation*}
m(T_{a}) \geq a- \frac{\norm{\Gamma}}{\kappa + \norm{\Gamma}} \stackrel{\eqref{L3.1.4}}{\geq} a-  \frac{2 (1- \kappa)}{2- \kappa}. 
\end{equation*}
Selecting $a_0 = 2 (1- \kappa) / (2- \kappa)$, we see that $T_{a_0}$ is a positive operator. Moreover, the application of the spectral theorem and the positivity of $\Gamma$ also show that $\rho (T_{ a_0}) \leq a_0$. If we now define $\sigma^2(d)= \kappa^{-1}(1-a_0) = 1/(2 - \kappa) \in \big(\frac{1}{2},1\big)$ (by \eqref{kappa}), and $G' = \kappa^{-1} T_{ a_0}$, so that $G'f(x)= \sum_{y \in K}g'(x,y)f(y)$ for $x \in K$ and $f \in \ell^2(K)$, then \eqref{L3.1.3} readily yields \eqref{3.1}. In addition, $G'$ is translation invariant, and by \eqref{1.17}, we see that $\sigma^2(d)$ tends to $1$ as $d \to \infty$, and the spectral radius of $G'$ satisfies
\begin{equation*} 
\rho(G') \leq \kappa^{-1} \rho(T_{a_0}) \leq \kappa^{-1} a_0 = \frac{2}{ \kappa(2-\kappa)} \cdot (1- \kappa) \stackrel{\eqref{1.17}}{\leq} c_3 / d, 
\end{equation*}
for a suitable constant $c_3 > 0$. This concludes the proof of Lemma \ref{L3.1}.
\end{proof}

We now decompose the Gaussian free field according to Lemma \ref{L3.1}. To this end, we let $\P_\psi$, $\P_\xi$, be probabilities on auxiliary probability spaces $\Omega_\psi$, $\Omega_\xi$, respectively endowed with random fields $(\psi_x)_{x \in \mathbb{Z}^3}$, $(\xi_x)_{x \in \mathbb{Z}^3}$, such that
\vspace{-0.5ex}
\begin{equation} \label{3.3}
\begin{split}
&\text{under $\P_\psi$, $(\psi_x)_{x \in \mathbb{Z}^3}$, is a centered Gaussian field with} \\
&\text{covariance $\E_\psi [\psi_x \psi_y] = \sigma^2 (d) \cdot \delta(x,y)$, for all $x,y \in \mathbb{Z}^3$},
\end{split}
\end{equation}

\vspace{-0.6ex}
\noindent and
\vspace{-0.6ex}
 \begin{equation} \label{3.4}
\begin{split}
&\text{under $\P_\xi$, $(\xi_x)_{x \in \mathbb{Z}^3}$, is a centered Gaussian field with} \\
&\text{covariance $\E_\xi [\xi_x \xi_y] =g'(x,y)$, for all $x,y \in \mathbb{Z}^3$}.
\end{split}
\end{equation}

\vspace{-0.6ex}
\noindent Then, by \eqref{3.1} and usual Gaussian field arguments (see \cite{AT}, p.11),
\vspace{-0.5ex}
\begin{equation} \label{DECOMP}
\text{$(\varphi_x)_{x \in \mathbb{Z}^3}$, under $\P$, has the same law as $(\psi_x  +  \xi_x)_{x \in \mathbb{Z}^3}$, under $\P_\psi \otimes \P_\xi$.}
\end{equation}
 Moreover, given any level $h \in \mathbb{R}$, we define the (random) sets 
 \vspace{-0.5ex}
\begin{equation} \label{3.5}
E_{\psi}^{\geq h} = \{ x \in \M \ ; \ \psi_x \geq h  \}, \qquad E_{\psi}^{< h} = \M \setminus E_{\psi}^{\geq h},
\end{equation}

\vspace{-0.8ex}
\noindent and $E_{\xi}^{\geq h}$, $E_{\xi}^{< h}$ in analogous manner. A crucial point is that the field $\xi$ acts only as a small perturbation when $d$ is large, which is entailed in \eqref{3.2} and more quantitatively in the following lemma, the proof of which uses ideas developed in \cite{MS} (see in particular Theorem 2.4 therein).

\begin{lemma} \label{L3.2} $\quad$
\smallskip

\noindent There exists a decreasing function $v: (c_3, \infty] \longrightarrow (0 ,1)$, with $\lim_{u \to \infty }v(u) = 0$, such that for all $h >0$ and $d \geq 6$ satisfying $h^2 > c_3 d^{-1}$, and all $A \subset \subset \M$,
\begin{equation} \label{3.6}
\P_\xi \Big[ \bigcap_{x \in A} \{  | \xi_x | > h \} \Big] \leq  \big[ v (h^2  d)\big]^{|A|}.
\end{equation}
\end{lemma} 

\begin{proof}
First note that
\begin{equation} \label{L3.2.1}
\P_\xi \Big[ \bigcap_{x \in A} \{  | \xi_x | > h \} \Big] \leq \P_\xi \Big[ \sum_{x \in A} \xi_x^2 > h^2 |A|  \Big], \qquad \text{for all } h >0.
\end{equation}
Now, assume some ordering of $A \subset \subset \M$ has been specified, and let $G'_A = \big(g'(x,y) \big)_{x,y \in A}$ denote the covariance matrix of the Gaussian vector $ \xi_A = (\xi_x)_{x \in A}$, with decreasing eigenvalues $ \lambda_i \geq 0$, $1 \leq i \leq |A|$, and write $\rho_A = \rho(G'_A)= \lambda_1$ for its spectral radius. Finally, define the diagonal matrix $\Lambda = \text{diag}(\{\lambda_i\})$. By spectral decomposition, $G'_A = O \Lambda O^T$ for some orthogonal matrix $O$. Let $\tilde{\xi}= (\tilde{\xi}_i)_{ 1 \leq i \leq |A|} $ be a Gaussian vector whose components are i.i.d. standard Gaussian variables, and $\P_{\tilde{\xi}}$ be its law. Then $ O \sqrt{\Lambda} \tilde{\xi} \sim \mathcal{N}(0,  O \Lambda O^T)$, i.e. $  O \sqrt{\Lambda} \tilde{\xi} \stackrel{d}{=} \xi_A$, and thus $\sum_{x \in A} \xi_x^2 \stackrel{d}{=} (O \sqrt{\Lambda} \tilde{\xi} )^T O \sqrt{\Lambda} \tilde{\xi} = \sum_{1 \leq i \leq |A|} \lambda_i \tilde{\xi}_i^ 2.$ Inserting this into \eqref{L3.2.1} yields
\begin{equation*}
\P_\xi \Big[ \bigcap_{x \in A} \{  | \xi_x | > h \} \Big] \leq \P_{\tilde{\xi}} \Big[ \sum_{1 \leq i \leq |A|} \tilde{\xi}_i^2 >  \rho_A^{-1} h^2 |A|  \Big] \leq \min_{0 < a <1} e^{-\frac{a h^2 |A|}{2 \rho_A}}  \mathbb{E}_{\tilde{\xi}} \Big[ \prod_{1 \leq i \leq |A|} e^{ \frac{a}{2} {\tilde{\xi}_i}^2} \Big],
\end{equation*}
where we have used Markov's inequality in the last step. But $ \mathbb{E}_{\tilde{\xi}}[e^{a{\tilde{\xi}_i}^2/2}] = (1-a)^{-1/2}$ for all $0<a<1$ and $1 \leq i \leq |A|$, thus yielding
\begin{equation} \label{L3.2.2}
\P_\xi \Big[ \bigcap_{x \in A} \{  | \xi_x | > h \} \Big] \leq \min_{0 < a <1} \Big[ \sqrt{1-a} \cdot e^{\frac{a h^2 }{2 \rho_A}} \Big]^{-|A|}.
\end{equation}
One easily verifies that the function $q(a) = \sqrt{1-a} \cdot e^{a h^2 / 2 \rho_A}$ attains a maximum in the interval $(0,1)$ only if $h^2 / \rho_A > 1 $, which certainly holds if $h^2 > c_3 d^{-1}$ by Lemma \ref{L3.1}. In this regime, the maximum is reached for $a= 1 - \rho_A / h^2$. Inserting this into \eqref{L3.2.2}, we obtain
\begin{equation*}
\P_\xi \Big[ \bigcap_{x \in A} \{  | \xi_x | > h \} \Big] \leq \Big[ \big(e  h^2\rho_A^{-1}\big)^{1/2} \cdot e^{- \frac{ h^2}{2 \rho_A}} \Big]^{|A|}.
\end{equation*}
The function $\tilde{v}(u) = (2eu)^{1/2}e^{-u}$ is $(0,1)$-valued and monotonically decreasing on $(1/2, \infty)$. Thus, $h^2 > c_3 d^{-1}$ ensures that $h^2 / 2\rho_A > 1/2$ and $\tilde{v}(h^2 / 2\rho_A) \leq \tilde{v}(h^2d/ 2c_3) \stackrel{\text{def.}}{=} v(h^2 d)$. This completes the proof of Lemma \ref{L3.2}. 
\end{proof}

We are now ready to introduce a central quantity before proceeding to the main theorem of this section. Namely, we define, for all $d \geq 6$, $h \in \mathbb{R}$ and positive integers $L_0$,
\begin{equation} \label{psi_slab}
\Psi^{(slab)} (d,h,L_0) =  \P \big[E_{\varphi}^{\geq h}  \cap \ \big( \mathbb{Z}^2 \times [0, 2L_0) \times \{ 0\}^{d-3} \big) \text{ contains an infinite cluster} \ \big].
\end{equation}

\vspace{0.1cm}

\begin{theorem} \label{T3.3} $\quad$
\smallskip

\noindent There exists $d_0 \geq 6$, a level $h_0 >0$ and an integer $L_0 \geq1$ such that
\begin{equation} \label{3.SlabPerc}
\Psi^{(slab)} (d,h_0,L_0) =1, \qquad \text{for all } d \geq d_0.
\end{equation}
In particular, the critical level $h_*(d)$ defined in \eqref{h*} satisfies 
\begin{equation} \label{3.8}
\quad h_*(d) \geq h_0 >0, \qquad \text{for all } d \geq d_0.
\end{equation}
\end{theorem}

\begin{proof} To begin with, we note that \eqref{3.8} immediately follows from \eqref{3.SlabPerc}. In order to prove \eqref{3.SlabPerc}, we will use the decomposition \eqref{DECOMP} of the Gaussian free field restricted to $\M$ and the bounds obtained in Lemma \ref{L3.2} to perform a Peierls-type argument.

We let $\PP_p = (p \delta_1 + (1-p) \delta_0)^{\otimes \M}$, for $p \in (0,1)$, and observe that the law of $\big(1\{ \psi_x \geq h \} \big)_{x \in \M}$ on $ \{ 0,1\}^{\M}$ (endowed with its canonical $\sigma$-algebra) is $\PP_{p(h, \sigma(d))}$, where
\begin{equation} \label{p_h_sigma}
p(h, \sigma) \stackrel{\text{def.}}{=} \frac{1}{\sqrt{2 \pi} \sigma}  \int_h^\infty e^{- x^2/2 \sigma^2} \text{d}x, \qquad \text{for } h, \sigma >0,
\end{equation}
so that $p(h, \sigma(d)) = \P_\psi [\psi_0 \geq h ]$ (see \eqref{3.3}). For arbitrary $h > 0$, $L_0 \geq 1$, $d \geq 6$ and $\omega_\xi \in \Omega_\xi$, we define the increasing event (part of $ \{ 0,1\}^{\M}$),
\begin{equation*}
A_\infty^h(\omega_\xi ) = \big\{ \omega \in \{ 0,1\}^{\M} \ ; \  \{ x \in \M ; \omega_x =1\} \cap E_\xi^{\geq -h}( \omega_\xi ) \cap (\mathbb{Z}^2 \times [0, 2L_0) \big) \text{ has an infinite cluster} \big\},
\end{equation*}
and obtain, for all $p' \leq p(2h, \sigma(d))$, using the decomposition \eqref{DECOMP},
\begin{equation} \label{Psi_0_1}
\begin{split}
\Psi^{(slab)} (d,h,L_0) 
&\geq \P_\psi \otimes \P_\xi \big[ E_\psi^{\geq 2h} \cap E_\xi^{\geq -h} \cap \big(\mathbb{Z}^2 \times [0, 2L_0) \big)\text{ contains an infinite cluster} \ \big] \\
&= \int_{\Omega_\xi} \text{d} \P_\xi(\omega_\xi ) \PP_{p(2h, \sigma(d))} \big[A_\infty^h (\omega_\xi )\big] \\
&\geq \int_{\Omega_\xi} \text{d} \P_\xi(\omega_\xi ) \PP_{p'} \big[A_\infty^h (\omega_\xi )\big],
\end{split} 
\end{equation}
where we have used that $A_\infty^h(\omega_\xi )$ is increasing in $\omega$ for every fixed $\omega_\xi$ in the last line. Since $\sigma^2(d) \geq 1/2$ for $d \geq 6$, we have $p(2h, \sigma(d)) \geq p(2h, 1/ \sqrt{2})=p'$ for all $h> 0$ and $d \geq 6$. Thus, \eqref{Psi_0_1} yields
\begin{equation} \label{Psi_0}
\Psi^{(slab)} (d,h,L_0) \geq \P_{\psi^0} \otimes \P_\xi \big[ E_{\psi^0}^{\geq 2h} \cap E_\xi^{\geq -h} \cap \big(\mathbb{Z}^2 \times [0, 2L_0) \big)\text{ contains an infinite cluster} \ \big],
\end{equation}
for all $d \geq 6$, $h > 0$ and $L_0 \geq 1$, where $\psi^0 = (\psi^0_x)_{x \in \M}$ is a field of independent centered Gaussian variables with variance $1/2$, i.e. as in \eqref{3.3} but with $1/2$ in place of $\sigma^2(d)$, and $\P_{\psi^0}$ denotes the probability on $(\Omega_{0}, \A_0)$ governing $\psi^0$. We will prove that the probability on the right-hand-side of \eqref{Psi_0} is equal to one for all $d \geq d_0$, $0< h \leq h_0$, with suitable $d_0$, $h_0$ and $L_0$. The claim \eqref{3.SlabPerc} will immediately follow from this.

We first construct certain families of ``good'' events for the Gaussian fields $\psi^0$ and $\xi$. To this end, we define boxes $B^{(3)}_x(L) = x + \big( [ 0 , L ) \cap \mathbb{Z} \big)^3$ for any $x \in \M$ and positive integer $L$, and introduce a renormalized lattice 
\begin{equation}
\LL= L_0 \mathbb{Z}^2 \ \; \big( \subset \M \big).
\end{equation}
We begin with $\psi^0$ and define $\mathcal{C}_{x} (\omega)$, for any $x \in \LL$ and $\omega \in  \{ 0,1 \}^{\M}$, to be the (possibly empty) open (i.e. value $1$) cluster in $B^{(3)}_x( 2L_0)$ containing the most vertices. If several such clusters exist, we choose one according to some given prescription (say using lexicographic order). For arbitrary $x \in \LL$, we introduce the event $F_{x}$ (in the canonical $\sigma$-algebra of $\{ 0,1\}^{\M}$) as follows: given some $\omega \in \{ 0,1\}^{\M}$, we let $\omega \in F_{x}$ if and only if
\begin{enumerate}
\item[i)] $\mathcal{C}_{x} (\omega)$ is a \textit{crossing} cluster for $B^{(3)}_x( 2L_0)$ in the first two axes-directions, i.e. for $i=1, \ 2,$ there exists an open path $\gamma_i$ in $\mathcal{C}_{x} (\omega)$ with endvertices $y_{(i)},z_{(i)}$ satisfying $y^i_{(i)} = x^i$ and $z^i_{(i)} = x^i + 2L_0 -1$. 
\item[ii)]$\mathcal{C}_{x} (\omega)$ is the \textit{only} open cluster $\mathcal{C}$ of $B^{(3)}_x( 2L_0)$ having the property $\text{diam}(\mathcal{C}) \geq L_0 -1$.
\end{enumerate}
Having introduced $F_{x}$ for all $n \geq 0$ and $x \in \LL$, we define, for $h>0$, the measurable map $\Phi_0^h : \Omega_0 \longrightarrow \{ 0,1\}^{\M},$ $\omega_0 \mapsto \big(1\{ \omega_{{0}_x} \geq h \} \big)_{x \in \M}$, and the events (in $\A_0$)
\begin{equation} \label{T3.3.6}
F_x^h = (\Phi_0^h)^{-1}(F_x), \qquad \text{for all } h>0 \text{ and } x \in \LL.
\end{equation} 
 
We now turn to the (``good'') events for $\xi$, set
\begin{equation} \label{T3.3.4}
G_{x}^{-h}  = \bigcap_{y \in B^{(3)}_x( 2L_0)} \{ \xi_y \geq -h \}, \qquad \text{for all } x \in \LL, \ h>0,
\end{equation}
and note that $\P_\xi [G_{x}^{-h} ] = \P_\xi [ G_{0}^{-h} ]$, for all $x \in \LL$, by translation invariance of $g'(\cdot, \cdot)$, see \eqref{3.1}.

With ``good'' events for $\psi^0$ and $\xi$ at hand (see \eqref{T3.3.6} and \eqref{T3.3.4}), for any $h >0$, we define a vertex $x \in \LL$ to be $h$-\textit{good} if the event
\begin{equation} \label{BAD}
F_{x}^{2h} \times G_{x}^{-h}
 \end{equation}
occurs (under $\P_{\psi^0} \otimes \P_{\xi}$), and $h$-\textit{bad} otherwise. A sequence $\gamma = (x_i)_{0 \leq i \leq m}$, $m \in \mathbb{N} \cup \{ \infty \}$, in $\LL$ such that $|x_{i+1}-x_i|= L_0$ for all $0 \leq i <  m$ will be called a \textit{nearest-neighbor path in $\LL$} (we will refer to $m$ as the length of the path). Similarly, a \textit{$*$-nearest-neighbor path in $\LL$} is any sequence in $\LL$ subject to the weaker condition $|x_{i+1}-x_i|_{\infty}= L_0$ for all $0 \leq i < m$.

A crucial property is that percolation of $h$-good sites in $\LL$ implies percolation of $ E_{\psi^0}^{\geq 2h} \cap E_\xi^{\geq -h}$ (in the slab $\mathbb{Z}^2 \times [0,2L_0)$), which we state as a

\begin{lemma} \label{L3.6} $\quad$
\smallskip

\noindent Let $h >0$, $\gamma = (x_i)_{0 \leq i \leq m}$, $m \in \mathbb{N} \cup \{ \infty \}$, be a nearest-neighbor path in $\LL$ and assume that all vertices $x_i$, $0 \leq i \leq m,$ are $h$-good. Then, the corresponding clusters $\mathcal{C}_{x_i} (\psi^0)$ $($pertaining to the events $F_{x_i}^{2h})$ are subsets of $E_{\psi^0}^{\geq 2h} \cap E_\xi^{\geq -h}$, which are all connected within the set $\bigcup_{i=0}^m B^{(3)}_{x_i}( 2L_0)$.
\end{lemma}
In particular, Lemma \ref{L3.6} implies that if $(x_i)_{i \geq 0}$ is an infinite nearest-neighbor path of $h$-good vertices in $\LL$ which is unbounded, then the set $\bigcup_{i=0}^\infty B^{(3)}_{x_i} (2L_0) \cap E_{\psi^0}^{\geq 2h} \cap E_\xi^{\geq -h}$  (a subset of $\mathbb{Z}^2 \times [0,2L_0)$) contains an infinite cluster.

\vspace{0.2cm}

\begin{proof5}
It suffices to consider the case $m=2$. The general case then follows by induction on $m$. Thus, let $x_1, x_2 \in \LL$, $|x_1 - x_2| = L_0$, be both $h$-good. The following holds for $i=1,2$: by definition of $F_{x_i}^{2h}$, the set  $E_{\psi^0}^{\geq 2h} \cap B^{(3)}_{x_i} (2L_0)$ contains a cluster $\mathcal{C}_{x_i}  \ \big( = \mathcal{C}_{x_i} (\psi^0) \ \big)$ which is crossing in the first two axes-directions. Moreover, since $G_{x_i}^{-h}$ occurs, $\xi_y \geq - h$ for all $y \in \mathcal{C}_{x_i}$, i.e. $\mathcal{C}_{x_i} \subset E_{\xi}^{\geq -h}$.

It remains to show the clusters $\mathcal{C}_{x_1}$ and $\mathcal{C}_{x_2}$ are connected within $B^{(3)}_{x_1}( 2L_0) \cup B^{(3)}_{x_2}( 2L_0)$. Let $k \in \{ 1, 2\}$ be such that $|x_1^k - x_2^k| = L_0$. Since  $\mathcal{C}_{x_1}$ is crossing for $B^{(3)}_{x_1}( 2L_0)$ in the $k$-th direction, $\text{diam}\big(\mathcal{C}_{x_1} \cap B^{(3)}_{x_2}( 2L_0)\big) \geq L_0 - 1$. Hence $\mathcal{C}_{x_1}$ and $\mathcal{C}_{x_2}$ are connected within $B^{(3)}_{x_2}( 2L_0)$ by condition ii) in the above definition of the events $F_{x}$. This concludes the proof of Lemma \ref{L3.6}. \hfill $\square$
\end{proof5}

\vspace{0.3cm}

We now carry on with the proof of Theorem \ref{T3.3}, and select the parameters $h_0$, $L_0$ and $d_0$. First note that $p_c^{\text{site}}(\M)$, the critical level for Bernoulli site percolation on $\M$, satisfies $p_c^{\text{site}}(\M)< 1/2$ (see \cite{CR}, Theorem 4.1). We may thus choose $h_0 > 0$ such that 
\begin{equation} \label{3.h0}
\P_{\psi^0}[\psi_{x=0}^0 \geq 2h_0] \stackrel{\eqref{p_h_sigma}}{=}p\Big(2 h_0, \frac{1}{\sqrt{2}}\Big) = \frac{1}{2} \Big(\frac{1}{2} + p_c^{\text{site}}(\M) \Big),
\end{equation}
which means that the Bernoulli site percolation model on $\M$ associated to choosing sites $x \in \M$ where $\psi_x^0 \geq 2h_0$, is supercritical. By the site-percolation version of Theorem 7.61 in \cite{Gr}, we thus obtain that 
\begin{equation} \label{sec3_1}
\lim_{L_0 \to \infty } \P_{\psi^0} \big[ F_{x}^{2h_0} \big] = 1, \qquad \text{for all } x \in \LL.
\end{equation}
Moreover, the collection $\big( 1 \{ F_{L_0y}^{2h_0} \} \big)_{y \in \mathbb{Z}^2}$ is $2$-dependent and, (see \cite{Gr}, Theorem 7.65, or \cite{Scho}, Theorem $0.0$), there exists a non-decreasing function $\pi: [0,1] \longrightarrow [0,1]$ with $\lim_{\delta \to 1}\pi(\delta) =1$ such that if $\P_{\psi^0} [ F_{0}^{2h_0}  ]  \geq \delta$, then
\begin{equation} \label{stochdom}
\begin{split}
&\big( 1 \{ F_{x}^{2h_0} \} \big)_{x \in \LL} \quad \text{stochastically dominates a family of independent} \\ 
&\text{Bernoulli random variables indexed by $\LL$, with success parameter $\pi(\delta)$}.
\end{split}
\end{equation}
Using \eqref{sec3_1}, we then choose $L_0$ the smallest positive integer such that
\begin{equation} \label{L0_sec3}
 \pi\big(\P_{\psi^0} [ F_{0}^{2h_0}  ] \big) \geq 1 -1/40.
\end{equation}
Having fixed $h_0$ and $L_0$, we choose, in the notation of Lemma \ref{L3.2}, a constant $c_4 > c_3$ such that
\begin{equation} \label{L3.4.1}
(2L_0)^3 \cdot v(u)^{1/4} \leq 1/40, \qquad \text{for all } u \geq c_4,
\end{equation}
(recall that $v(\cdot)$ is monotonically decreasing) and define
 \begin{equation} \label{d_0def}
 d_0 =[ h_0^{-2} c_4] +1, 
\end{equation}
so that \eqref{L3.4.1} and \eqref{d_0def} yield
\begin{equation} \label{3.ProbaD0}
(2L_0)^3 \cdot  v(h_0^2 d)^{1/4} \leq (2L_0)^3 \cdot v(h_0^2 d_0)^{1/4} \leq 1/40, \qquad \text{for all } d \geq d_0.
\end{equation}

We now proceed to the last step of the proof, which mainly encompasses a Peierls argument. To this end, we introduce, for $x \in \LL $ and $N$ a multiple of $L_0$, the event $H^{h_0}(x,N)$  that $x$ is connected to $\{ y \in \LL ; |y-x|_\infty= N \}$, the restriction to $\LL$ of the $\ell^\infty$-sphere of radius $N$ centered around $x$, by a $*$-path of $h_0$-bad vertices in $\LL$. In the following lemma, we show that this event has small probability.

\begin{lemma} \label{PeierlsBound} $\quad$
\smallskip

\noindent For all $d \geq d_0$, $n \geq 1$, and $x \in \LL$,
\begin{equation} \label{3.Peierls}
\P_{\psi^0} \otimes \P_{\xi} \big[ H^{h_0}(x,nL_0) \big] \leq 2^{-n}.
\end{equation}
\end{lemma}

\begin{proofP}
For $x \in \LL$ and $n \geq 1$, we denote by $\Gamma_{x,n}^*$ the set of self-avoiding $*$-paths in $\LL$ starting at $x$ of length $n$. For $H^{h_0}(x,nL_0)$ to occur, there must be a self-avoiding $*$-path $\gamma=(x_i)_{0\leq i \leq n}$ of $h_0$-bad vertices in $\LL$ starting at $x$, hence
\begin{equation} \label{LL1}
\begin{split}
\P_{\psi^0} \otimes \P_{\xi} \big[ H^{h_0}(x,nL_0) \big]  
&\leq\P_{\psi^0} \otimes \P_{\xi} \Big[ \bigcup_{\gamma  \in\Gamma_{x,n}^* } \{ \gamma \text{ is $h_0$-bad}\}\Big] \\
& \leq |\Gamma_{x,n}^*| \sup_{\gamma = (x_i)_{0\leq i \leq n} \in \Gamma_{x,n}^*} \P_{\psi^0} \otimes \P_{\xi} \Big[ \bigcap_{i=1}^n \Big( (F_{x_i}^{2h_0})^c  \cup  (G_{x_i}^{-h_0})^c \Big) \Big],
\end{split}
\end{equation}
where we have identified $(F_{x_i}^{2h_0})^c$ with $(F_{x_i}^{2h_0})^c \times \Omega_\xi$ in the last line (and similarly for $(G_{x_i}^{-h_0})^c$). For arbitrary $\gamma = (x_i)_{0\leq i \leq n} \in \Gamma_{x,n}^*$, the probability on the right-hand side of \eqref{LL1} is equal to (setting $[[n]] = \{1,\dots,n \}$)
\begin{equation} \label{LL2}
\begin{split}
&\P_{\psi^0} \otimes \P_{\xi} \Big[ \bigcup_{k=0}^n \  \bigcup_{\{i_1, \dots, i_k \} \subset [[n]]} \  \bigcap_{j \in [[n]] \setminus \{i_1, \dots, i_k \}}  (F_{x_j}^{2h_0})^c \ \bigcap_{i\in \{i_1, \dots, i_k \}} (G_{x_i}^{-h_0})^c   \Big] \\
&\leq \sum_{k=0}^n \binom{n}{k} \sup_{\{i_1, \dots, i_k \} \subset [[n]]} \P_{\psi^0}   \Big[   \bigcap_{j \in [[n]] \setminus \{i_1, \dots, i_k \}}  (F_{x_j}^{2h_0})^c \Big] \cdot \P_{\xi} \Big[ \bigcap_{i\in \{i_1, \dots, i_k \}} (G_{x_i}^{-h_0})^c \Big].
\end{split}
\end{equation}
By stochastic domination and our choice of $L_0$, c.f. \eqref{stochdom} and \eqref{L0_sec3}, we have
\begin{equation} \label{LL3}
\P_{\psi^0}   \Big[   \bigcap_{j \in [[n]] \setminus \{i_1, \dots, i_k \}}   (F_{x_j}^{2h_0})^c \Big] \leq 40^{-(n-k)}.
\end{equation}
When $(G_{x_{i}}^{-h_0})^c$, $i=i_1, \dots, i_k$, simultaneously occur, we can choose $k$ sites $z_j$, $1 \leq j \leq k$, in the respective boxes $B_{x_{i_j}}^{(3)}(2L_0)$, $1 \leq j \leq k$, such that $\xi_{z_j} < -h_0$ (c.f. definition \eqref{T3.3.4}). Any such $z_j$ belongs to exactly four boxes $B_x^{(3)}(2L_0)$, $x \in \LL$. Since $(x_i)_{0 \leq i \leq n}$ is self-avoiding in $\LL$, we thus have $|\{ z_j \ ; \ 1 \leq j \leq k\}| \geq k/4$. As a result, Lemma \ref{L3.2} yields
\begin{equation} \label{LL4}
\P_{\xi} \Big[ \bigcap_{i\in \{i_1, \dots, i_k \}} (G_{x_i}^{-h_0} )^c \Big] \leq (2L_0)^{3k} \cdot \big( v(h_0^2d_0) \big)^{\frac{k}{4}} \stackrel{\eqref{3.ProbaD0}}{\leq} 40^{-k}.
\end{equation}
Note that $|\Gamma_{x,n}^*| \leq (3^2-1)^n=8^n$. Putting \eqref{LL2}, \eqref{LL3} and \eqref{LL4} together, and substituting the resulting bound into \eqref{LL1}, we finally obtain,
\begin{equation*}
\P_{\psi^0} \otimes \P_{\xi} \big[ H^{h_0}(x,nL_0) \big] \leq  8^n  \sum_{k=0}^n \binom{n}{k} 40^{-k}40^{-(n-k)}=( 2 / 5 )^n < 2^{-n}.
\end{equation*}
This completes the proof of Lemma \ref{PeierlsBound}. \hfill $\square$
\end{proofP}

\vspace{0.3cm}

We can now conclude the proof of Theorem \ref{T3.3}. For arbitrary $d \geq d_0$, we consider the event that the set $\LL \cap [0,2L_0)^2$ is surrounded by a $*$-circuit (i.e. a self-avoiding $*$-path except for the end point which coincides with the starting point) of $h_0$-bad vertices in $\LL$. Considering a point of this circuit on the first axis with largest coordinate, we see that the probability of this event is bounded by
\begin{equation} \label{T3.3.9}
\sum_{n=2}^\infty \P_{\psi^0} \otimes \P_{\xi} [H^{h_0}(0,nL_0)] \stackrel{\eqref{3.Peierls}}{\leq}  \sum_{n=2}^\infty 2^{-n} < 1.
\end{equation}
If this event does not occur, then by planar duality (c.f. \cite{Gr}, Section 11.2), there exists an infinite self-avoiding nearest-neighbor path $\gamma =(x_i)_{i\geq 0}$ of $h_0$-good vertices in $\LL$. Lemma \ref{L3.6} (see in particular the remark following it) then implies that the set
\begin{equation*}
 \Big( \bigcup_{i=0}^\infty B^{(3)}_{x_i} (2L_0) \cap E_{\psi^0}^{\geq 2h_0} \cap E_\xi^{\geq -h_0} \Big) \quad  \subset \  \mathbb{Z}^2 \times [0,2L_0)  
\end{equation*}
contains an infinite cluster. By \eqref{T3.3.9}, this event happens with positive probability, so that
\begin{equation*}
\P_{\psi^0} \otimes \P_\xi \big[ E_{\psi^0}^{\geq 2h_0} \cap E_\xi^{\geq -h_0} \cap \big(\mathbb{Z}^2 \times [0, 2L_0) \big)\text{ contains an infinite cluster} \ \big] > 0, \ \text{for all } d \geq d_0.
\end{equation*}
It then follows by \eqref{Psi_0} and ergodicity (see Lemma \ref{perc_prop}) that the value of $\Psi^{(slab)} (d,h_0,L_0)$ is in fact one. This proves \eqref{3.SlabPerc} and thus concludes the proof of Theorem \ref{T3.3}.
\end{proof}

\begin{remark} $\quad$ \label{ConcRems}
\medskip

\noindent1) We show in Theorem \ref{T3.3} that $E^{\geq h}_\varphi$ percolates in a two-dimensional slab for small but positive $h$ when $d$ is sufficiently large. However, it should be underlined that $E^{\geq h}_\varphi \cap \mathbb{Z}^2$ does not percolate for any $h \geq 0$, as we now briefly explain. Indeed, the conditions of Theorem 14.3 in \cite{HJ} (which is itself a variant of \cite{GKR} when the finite energy condition holds) are met for the law of $\big( 1\{ \varphi_x \geq 0 \}\big)_{x \in \mathbb{Z}^2}$ on $\{0,1 \}^{\mathbb{Z}^2}$ under $\P$ (in particular, positive correlations, see Definition 14.1 in \cite{HJ}, follow from the FKG-inequality). Hence $E^{\geq 0}_\varphi \cap \mathbb{Z}^2$ and its complement  in $\mathbb{Z}^2$ cannot \emph{both} have infinite clusters. Observe that $\big( 1\{ \varphi_x \geq 0 \}\big)_{x \in \mathbb{Z}^2}$ and $\big( 1\{ \varphi_x < 0 \}\big)_{x \in \mathbb{Z}^2}$ have the same law under $\P$, by symmetry. If $E^{\geq 0}_\varphi \cap \mathbb{Z}^2$ percolated (with probability one, by ergodicity), the same would hold true as well for $E^{< 0}_\varphi \cap \mathbb{Z}^2$, leading to a contradiction. Therefore, $E^{\geq 0}_\varphi \cap \mathbb{Z}^2$ does not percolate.

\vspace{0.3cm}

\noindent2) With Lemma \ref{L3.2} at hand, one may immediately apply the criterion of Molchanov and Stepanov (c.f. \cite{MS}, Theorem 2.1) to infer that $E_\xi^{\geq -h}$ (c.f. \eqref{3.5} for notation) percolates \textit{strongly} in $\M$ when $h^2d$ is sufficiently large, i.e. not only does $E_\xi^{\geq -h}$ percolate, but in addition $E_\xi^{< -h}$ doesn't, $\P_{\xi}$-almost surely. However, we note that $E_\psi^{\geq h}$ does not percolate strongly, since for all $h >0$, we have $p_c = p_c^{\text{site}}(\M) < p(h,\sigma(d)) = \P_\psi [\psi_0 \geq h] < 1/2$,  hence in particular $p(h,\sigma(d)) \in (p_c,1-p_c)$, where both $E_\psi^{\geq h}$ and its complement possess an infinite cluster with positive probability (in fact with probability one). 
\vspace{0.3cm}

\noindent3) It remains open whether $h_*(d)$ is actually strictly positive for \textit{all} $d \geq 3$. Recent simulations suggest this is the case when $d=3$, with an approximate value $\P[\varphi_0 \geq h_*] \simeq 0.16$, see \cite{M}, Section 4.1.2, and Figure 4.1 in Appendix 4.4. \hfill $\square$
\end{remark}

\vspace{-0.4cm}

\section*{Acknowledgments}

Pierre-Fran\c cois Rodriguez would like to thank Bal\'azs R\'ath and Art\"em Sapozhnikov for presenting their recent work \cite{RS}, as well as David Belius and Alexander Drewitz for stimulating discussions.

\end{document}